\documentclass[10pt]{amsart}

\usepackage[english]{babel}
\usepackage{amsthm,amsfonts,amsmath,amssymb}
\usepackage{hyperref}
\usepackage{color}
\usepackage{mathrsfs}
\usepackage[all]{xy}

\usepackage{hyperref}

\theoremstyle{plain}
\newtheorem{lemma}{Lemma}[section]
\newtheorem{theorem}[lemma]{Theorem}
\newtheorem{proposition}[lemma]{Proposition}
\newtheorem{corollary}[lemma]{Corollary}

\newtheorem{theorem-var}{Theorem}[]
\newtheorem{corollary-var}{Corollary}[]

\theoremstyle{definition}
\newtheorem{definition}[lemma]{Definition}
\newtheorem*{definition*}{Definition}

\theoremstyle{remark}
\newtheorem{remark}[lemma]{Remark}
\newtheorem{example}[lemma]{Example}

\DeclareMathOperator{\Root}{Root}

\DeclareMathOperator{\Div}{Div}
\DeclareMathOperator{\Aut}{Aut}
\DeclareMathOperator{\Hom}{Hom}
\DeclareMathOperator{\Lie}{Lie}
\DeclareMathOperator{\conv}{conv}
\DeclareMathOperator{\cone}{cone}

\DeclareMathOperator{\supp}{supp}

\newcommand{\calG}{\mathcal G}

\newcommand{\calX}{\mathcal X}

\newcommand{\mC}{\mathbb C} \newcommand{\mN}{\mathbb N}
 \newcommand{\mQ}{\mathbb Q}
\newcommand{\mZ}{\mathbb Z}

\newcommand{\gog}{\mathfrak g}

\newcommand{\goh}{\mathfrak h}

\newcommand{\gou}{\mathfrak u}
\newcommand{\gon}{\mathfrak n}

\newcommand{\gor}{\mathfrak r}
\newcommand{\got}{\mathfrak t}

\newcommand{\sfA}{\mathsf A} \newcommand{\sfB}{\mathsf B}
\newcommand{\sfC}{\mathsf C} 
 \newcommand{\sfF}{\mathsf F}
\newcommand{\sfG}{\mathsf G}

\newcommand{\scrB}{\mathscr B}
\newcommand{\scrC}{\mathscr C}
\newcommand{\scrD}{\mathscr D}
\newcommand{\scrF}{\mathscr F}

\newcommand{\scrO}{\mathscr O}

\newcommand{\scrS}{\mathscr S}

\newcommand{\scrU}{\mathscr U}

\newcommand{\gra}{\alpha} \newcommand{\grb}{\beta}    \newcommand{\grg}{\gamma}
\newcommand{\grd}{\delta} \newcommand{\grl}{\lambda}  \newcommand{\grs}{\sigma}

 \newcommand{\grD}{\Delta}  \newcommand{\grL}{\Lambda}
 \newcommand{\grS}{\Sigma}

\newcommand{\ra}         {\rightarrow}
\newcommand{\lra}        {\longrightarrow}

\newcommand{\vuoto}      {\varnothing}

\renewcommand{\geq}      {\geqslant}
\renewcommand{\leq}      {\leqslant}
\newcommand{\senza}      {\smallsetminus}

\newcommand{\ol}         {\overline}

            \newcommand{\st}       {\, : \,}

\newcommand{\mru}       {\mathrm u}

\DeclareMathOperator{\Stab}{Stab}
\DeclareMathOperator{\rk}{rk}

\newcommand{\weak}{\Psi^\mathrm{\sharp}}
\newcommand{\M}{\mathrm M}
\newcommand{\m}{\mathrm m}

\begin{document}

\title[Orbits of strongly solvable spherical subgroups]{Orbits of strongly solvable spherical subgroups\\on the flag variety}

\author{Jacopo Gandini, Guido Pezzini}

\email{jacopo.gandini@sns.it}

\curraddr{\textsc{Scuola Normale Superiore\\ Piazza dei Cavalieri, 7\\ 56126 Pisa, Italy}}

\email{pezzini@mat.uniroma1.it}

\curraddr{\textsc{Dipartimento di Matematica\\ Sapienza Universit\`a di Roma\\ Piazzale Aldo Moro 5\\ 00185 Roma, Italy}}

\subjclass[2010]{14M10, 14M27}
\keywords{Strongly solvable spherical subgroups, orbits of a Borel subgroup, weight polytopes}

\begin{abstract}
Let $G$ be a connected reductive complex algebraic group and $B$ a Borel subgroup of $G$. We consider a subgroup $H \subset B$ which acts with finitely many orbits on the flag variety $G/B$, and we classify the $H$-orbits in $G/B$ in terms of suitable root systems. As well, we study the Weyl group action defined by Knop on the set of $H$-orbits in $G/B$, and we give a combinatorial model for this action in terms of weight polytopes.
\end{abstract}

\maketitle

\section*{Introduction}

Let $G$ be a connected complex reductive algebraic group and let $B \subset G$ be a Borel subgroup. A subgroup $H \subset G$ is called {\em spherical} if it has an open orbit on the flag variety $G/B$. If this is the case, as independently proved by Brion \cite{Br1} and Vinberg \cite{Vi}, then $H$ acts with finitely many orbits on $G/B$.

The best known example is that of $B$ itself, and more generally any parabolic subgroup of $G$. The $B$-orbits in $G/B$ are indeed the Schubert cells, and are finitely many thanks to the Bruhat decomposition. Another well studied case is that of the {\em symmetric subgroups} of $G$, i.e.\ when $H$ is the subgroup of fixed points of some algebraic involution of $G$. Especially in this case, the study of the $H$-orbits on $G/B$, their classification and the geometry of their closures are important in representation theory (see e.g.\ \cite{Sp1}, \cite{Wo}). An equivalent problem is the study of the $B$-orbits in $G/H$, and the geometry of their closures: they are fundamental objects to understand the topology of $G/H$ and of its embeddings (see \cite{FMSS}).

Spherical subgroups are classified in combinatorial terms, see \cite{Kr}, \cite{Br0}, \cite{Mi}, \cite{Lu1}, \cite{Lu2} where several particular classes of subgroups are considered, and the more recent papers \cite{Lo}, \cite{BP}, \cite{CF} where the classification is completed in full generality. Nevertheless the set $\scrB(G/H)$ of the $B$-orbits in $G/H$ is still far from being understood, essentially except for the cases of the parabolic subgroups and of the symmetric subgroups of $G$ (the latter especially thanks to the work of Richardson and Springer \cite{RS1}, \cite{RS2}).

The goal of the present paper is to explicitly understand the set $\scrB(G/H)$ in some other case, and to produce some combinatorial model for it. More precisely, we consider the case of the \textit{strongly solvable} spherical subgroups of $G$, that is, spherical subgroups of $G$ which are contained in a Borel subgroup. As a consequence of a theorem of Brion (see \cite[Theorem 6]{Br3}), under these assumptions the $H$-orbit closures in $G/B$ provide nice generalizations of the Schubert varieties: even though $H$ might be not connected, they are always irreducible, and they have rational singularities, so that in particular they are normal and Cohen-Macaulay.

When $H$ is a strongly solvable spherical subgroup, the set $\scrB(G/H)$ has already been studied in the literature in some special cases: Timashev \cite{Ti1} treated the case where $H = TU'$ (where $T$ denotes a maximal torus of $B$ and $U'$ the derived subgroup of the unipotent radical of $B$), and Hashimoto \cite{Ha} treated the case where $G = \mathrm{SL}_n$ and $H$ is a Borel subgroup of $\mathrm{SL}_{n-1}$, regarded as a subgroup of $\mathrm{SL}_n$. 

Let $\Phi$ be the root system of $G$ and let $W$ be its Weyl group. Given a strongly solvable spherical subgroup $H \subset G$, in our main result we give an explicit parametrization of the set $\scrB(G/H)$ by attaching to every $B$-orbit an element of $W$ and a root subsystem of $\Phi$. To do this, we build upon known results about the classification of strongly solvable subgroups, which is available in three different forms. The first one was given by Luna in 1993 (see \cite{Lu1}), the second one emerged in the framework of the general classification of spherical subgroups, and the third one, more explicit, has been given recently by Avdeev in \cite{Avd1} (see \cite{Avd2} for a comparison between the three approaches).

Our results on $\scrB(G/H)$ also provide a nice description of the action of $W$ on $\scrB(G/H)$, defined by Knop in \cite{Kn} for any spherical subgroup $H\subset G$. While the simple reflections of $W$ act in a rather explicit way, the resulting action of the entire $W$ is quite difficult to study. When $H$ is strongly solvable, we will see how this action becomes actually very simple: the fact that two $B$-orbits are in the same $W$-orbit will boil down to the fact that the associated root subsystems are $W$-conjugated. This will enable us to give a simple combinatorial model for $\scrB(G/H)$ as a finite set endowed with an action of $W$ in terms of ``generalized faces'' of weight polytopes.

We now explain our results in more detail. Fix a maximal torus $T \subset B$ and denote by $U$ the unipotent radical of $B$. We also denote by $\Phi$ the root system associated with $T$, by $W = N_G(T)/T$ its Weyl group, by $\Phi^+ \subset \Phi$ the set of positive roots associated with $B$ and by $\grD \subset \Phi^+$ the corresponding basis of $\Phi$. Let $H$ be a strongly solvable spherical subgroup of $G$, up to conjugation we may assume $H \subset B$. Up to conjugation by an element of $B$, we may assume as well that $T \cap H$ is a maximal diagonalizable subgroup of $H$. Given $\gra \in \Phi^+$, let $U_\gra \subset U$ be the associated unipotent one dimensional subgroup.

Two natural sets which are attached to $H$ are the set of \textit{active roots}
$$
	\Psi = \{ \gra \in \Phi^+ \st U_\gra \not \subset H\},
$$
introduced by Avdeev in \cite{Avd1}, and the corresponding set of restricted characters
$$\ol \Psi = \{\gra_{|{T \cap H}} \st \gra \in \Psi\}.$$
From a geometrical point of view, $\ol\Psi$ is canonically identified with the set of the $B$-stable prime divisors in $G/H$ which map dominantly to $G/B$ via the natural projection (see Section \ref{s:BmodH}). Given $I \subset \ol\Psi$, let $\Psi_I \subset \Psi$ be the subset of those roots $\gra$ such that $\gra_{|{T \cap H}} \in I$ and set
$$\Phi_I= \mZ \Psi_I \cap \Phi.$$
This is a parabolic subsystem of $\Phi$ which is explicitly described once $H$ is described in terms of active roots following Avdeev's classification, and the intersections $\Phi^\pm_I = \Phi_I \cap \Phi^\pm$ define a subdivision of $\Phi_I$ into positive and negative roots.

Given $w \in W$ and $I \subset \ol\Psi$, we say that $(w,I)$ is a \textit{reduced pair} if $w(\Phi^+_I) \subset \Phi^-$. Our first main result is the following.

\begin{theorem-var}[Corollary~\ref{cor:parametrizzazione}]\label{thm-var:bij}
Let $H$ be a spherical subgroup of $G$ contained in $B$. There is a natural bijection between the set of reduced pairs and the set of $B$-orbits in $G/H$.
\end{theorem-var}

Given a reduced pair $(w,I)$ we denote by $\scrO_{w,I}$ the corresponding $B$-orbit in $G/H$. Reduced pairs encode important properties of the corresponding $B$-orbit, and they can be used to encode combinatorial properties of the entire set $\scrB(G/H)$ as well. Denote by $W_I$ the Weyl group of $\Phi_I$, canonically embedded in $W$. Our next main result describes in these terms the action of $W$ on $\scrB(G/H)$.

\begin{theorem-var}[Theorems~\ref{teo:stabilizzatori1} ~\ref{teo:stabilizzatori2} and Corollary~\ref{cor:numero-Borbite}]\label{thm-var:W}
Let $H$ be a spherical subgroup of $G$ contained in $B$ and let $(w,I)$ be a reduced pair, then $\Stab_W(\scrO_{w,I}) = w W_I w^{-1}$. Moreover two orbits $\scrO_{w,I}$ and $\scrO_{v,J}$ are in the same $W$-orbit if and only if $I = J$, in which case $\scrO_{w,I} = wv^{-1} \cdot \scrO_{v,I}$. In particular, the following formula holds:
$$
	|\scrB(G/H)| = \sum_{I \subset \ol\Psi} |W/W_I|.
$$
\end{theorem-var}

A very special example of a strongly spherical subgroup was treated by Timashev in \cite{Ti1}, where the corresponding set of $B$-orbits is studied. More precisely, let $U'$ be the derived subgroup of $U$, namely $U' = \prod_{\Phi^+ \senza \grD} U_\gra$, and consider the subgroup $TU' \subset G$: this is a spherical subgroup of $G$ contained in $B$, and we have equalities $\ol\Psi = \Psi = \grD$. In this case the parametrization of the $B$-orbits in terms of reduced pairs can be proved in a simple way by using the commutation relations among root subgroups, and the set of reduced pairs $(w,I)$ is easily seen to be in a $W$-equivariant bijection with the set of faces of the weight polytope associated with any given dominant regular weight of $G$. This elegant description generalizes to the case of any strongly solvable spherical subgroup as follows.

Let $\grl$ be a regular dominant weight and let $P = \conv(W\grl)$ be the associated weight polytope in $\grL_\mQ = \Lambda \otimes_\mZ \mQ$, where $\Lambda$ denotes the weight lattice of $G$. Since $\grl$ is regular, the elements $w\grl$ with $w \in W$ are all distinct and coincide with the vertices of $P$. By a \textit{subpolytope} of $P$ we mean the convex hull of a subset of vertices of $P$. Notice that $P$ is naturally endowed with an action of $W$ which permutes the subpolytopes of $P$, and we denote by $\scrS(P)$ the set of subpolytopes of $P$.

Given a spherical subgroup $H \subset G$ contained in $B$, to any reduced pair $(w,I)$ we may associate a subpolytope of $P$ by setting $\scrS_{w,I} = \conv(wW_I\grl)$. This enables us to reformulate our combinatorial model of $\scrB(G/H)$ as follows.

\begin{theorem-var}[Theorem~\ref{teo:subpolytopes}]\label{teo:injS}
Let $H$ be a spherical subgroup of $G$ contained in $B$, then the map $(w,I) \mapsto \scrS_{w,I}$ defines a $W$-equivariant injective map $\scrB(G/H) \longrightarrow \scrS(P)$. Moreover, $\scrS_{w,I}$ is described as the intersection of $P$ with a cone in $\grL_\mQ$ as follows
$$
	\scrS_{w,I} = P \cap \big(w\grl+\mQ_{\geq0}(w(\Phi^-_I)) \big).
$$
\end{theorem-var}

The map $(w,I) \mapsto \scrS_{w,I}$ is not surjective in general, but applying the last part of the previous theorem it is possible in any given example to describe explicitly the image of the map. Moreover, this map has the advantage of being compatible with the Bruhat order, i.e.\ the inclusion relation between $B$-orbit closures: if $(v,J)$ and $(w,I)$ are reduced pairs such that $\scrS_{v,J} \subset \scrS_{w,I}$, then $\scrO_{v,J} \subset \ol{\scrO_{w,I}}$ (see Proposition~\ref{prop: bruhat-compatibile}). Unfortunately the converse of the previous statement is false in general, and a complete description of the Bruhat order of $\scrB(G/H)$ remains an open problem.

In case $H = TU'$ is easy to see that the image of the map $(w,I) \mapsto \scrS_{w,I}$ equals the set of faces of $P$. It was conjectured by Knop that $G/TU'$ has the largest number of $B$-orbits among all the spherical homogeneous spaces for $G$. Using our formula for the cardinality of $\scrB(G/H)$ we prove this in the solvable case.

\begin{theorem-var} [Theorem~\ref{teo:knop-conj}]\label{thm-var:knopconj}
The spherical homogeneous space $G/TU'$ has the largest number of $B$-orbits among the homogeneous spherical varieties $G/H$ with $H$ a solvable subgroup of $G$.
\end{theorem-var}

We explain now the structure of the paper. In Section~\ref{s:preliminaries} we explain our notations and collect some basic facts about spherical varieties and toric varieties. Then we restrict to the case of a spherical subgroup $H \subset G$ contained in $B$, and in Section~\ref{s:toricvar} we study the variety $B/H$: this is an affine toric $T$-variety whose set of $T$-stable prime divisors is naturally parametrized by $\ol\Psi$, and whose set of $T$-orbits is naturally parametrized by the subsets of $\ol\Psi$. In Section~\ref{s:active} we introduce the notion of \textit{weakly active root}. These are the roots $\gra \in \Phi^+$ whose associated root subgroup $U_\gra \subset B$ acts non-trivially on $B/H$, and we use them to attach a root system $\Phi_I$ to every subset $I \subset \ol\Psi$, thus to every $T$-orbit in $B/H$. In Section~\ref{s:Borbits} we introduce the notion of reduced pair (and the analogous one of \textit{extended pair}), and we prove Theorem~\ref{thm-var:bij}. Finally, in Section~\ref{s:W} we study the action of $W$ on $\scrB(G/H)$ and we prove Theorems~\ref{thm-var:W} and \ref{teo:injS}, and in Section~\ref{s:Knop-bound} we prove Theorem \ref{thm-var:knopconj}.\\

\textit{Acknowledgements.} We thank P.~Bravi, M.~Brion, F.~Knop and A.~Maffei for useful conversations on the subject, and especially R.S.~Avdeev for numerous remarks and suggestions on previous versions of the paper which led to significant improvements. 

This work originated during a stay of the first named author in Friedrich-Alexander-Universit\"at Erlangen-N\"urnberg during the fall of 2012 partially supported by a DAAD fellowship, he is grateful to F.~Knop and to the Emmy Noether Zentrum for hospitality. Both the authors were partially supported by the DFG Schwerpunktprogramm 1388 -- Darstellungstheorie.

%%%%%%%%%%%%%%%%%%%%%%%%%%%%%%%%%%%%
%%%%%%%%%%%%%%%%%%%%%%%%%%%%%%%%%%%%
\section{Notations and preliminaries}\label{s:preliminaries}
%%%%%%%%%%%%%%%%%%%%%%%%%%%%%%%%%%%%
%%%%%%%%%%%%%%%%%%%%%%%%%%%%%%%%%%%%

%%%%%%%%%%%%%%%%%%%%%%%%%%%%%%%%%%%%%%%%%%%%%%%
\subsection{Generalities}
%%%%%%%%%%%%%%%%%%%%%%%%%%%%%%%%%%%%%%%%%%%%%%%

All varieties and algebraic groups that we will consider will be defined over the complex numbers. Let $G$ be a connected reductive algebraic group. Given a subgroup $K \subset G$, we denote by $\calX(K)$ the group of characters of $K$ and by $K^\mru$ the unipotent radical of $K$. The Lie algebra of $K$ will be denoted either by $\Lie K$ or by the corresponding fraktur letter (here $\mathfrak k$). Given $g \in G$ we set $K^g = g^{-1}Kg$ and ${}^g\!K = gKg^{-1}$. If $K$ acts on an algebraic variety $X$, we denote by $\Div_K(X)$ the set of $K$-stable prime divisors of $X$. If $S$ is a group acting on an algebraic variety $X$ and if $x \in X$, we denote by $\Stab_S(x)$ the stabilizer of $x$ in $S$.  If $M$ is a lattice (that is, a free and finitely generated abelian group), the dual lattice is $M^\vee = \Hom_{\mZ}(M,\mZ)$, and the corresponding $\mQ$-vector space $M \otimes_{\mZ} \mQ$ is denoted by $M_\mQ$.

Let $B$ be a connected solvable group and let $T \subset B$ be a maximal torus. Suppose that $Z$ is a homogeneous $B$-variety. The \textit{weight lattice} of $Z$ is the lattice
$$
\calX_B(Z) = \{\text{$B$-weights of non-zero rational $B$-eigenfunctions } f \in \mC(Z) \},
$$
and the \textit{rank} of $Z$ is by definition the rank of its weight lattice. When the acting group is clear from the context, we will drop the subscript and write simply $\calX(Z)$. By $\mC(Z)^{(B)}$ we will denote the set of non-zero rational $B$-eigenfunctions of $Z$. We say that $z \in Z$ is a \textit{standard base point} for $Z$ (with respect to $T$) if $\Stab_T(z)$ is a maximal diagonalizable subgroup of $\Stab_B(z)$. Notice that standard base points always exist: indeed, since $B$ is connected, every diagonalizable subgroup of $B$ is contained in a maximal torus, and the maximal tori of $B$ are all conjugated (see \cite[Theorem 19.3 and Proposition 19.4]{Hu}). Notice also that, if $z \in Z$ is standard, then every $z' \in Tz$ is standard as well.

\begin{lemma}	\label{lemma:B-rank}
Let $Z$ be a homogeneous $B$-variety and let $z_0 \in Z$ be a standard base point, then $T z_0$ is a closed $T$-orbit. If moreover $H = \Stab_B(z_0)$, then $H = (T \cap H) H^\mru$ and the followings equalities hold:$$\calX_B(Z) = \calX(B)^H = \calX_T(Tz_0)$$
(where $\calX(B)^{H}$ denotes the subgroup of $\calX(B)$ of characters that are trivial on $H$).
\end{lemma}

\begin{proof}
For all $z \in Z$ the stabilizer $\Stab_T(z)$ is a diagonalizable subgroup of $\Stab_B(z)$. On the other hand $Z$ is homogeneous, therefore all the stabilizers $\Stab_B(z)$ are isomorphic and the maximal dimension for $\Stab_T(z)$ is the dimension of a maximal torus of $\Stab_B(z)$. By definition $\Stab_T(z_0)$ contains a maximal torus of $\Stab_B(z)$, therefore $T z_0$ has minimal dimension in $Z$, hence it is closed. The last claim is immediate.
\end{proof}

From now on $B$ will denote a Borel subgroup of $G$, $T\subset B$ a maximal torus, and $U = B^\mru$ the unipotent radical of $B$. Let $\Phi \subset \calX(T)$ be the root system of $G$ associated with $T$, $\Phi^+$ (resp.\ $\Phi^-$) the set of positive (resp.\ negative) roots determined by $B$ and $\grD \subset \Phi^+$ the corresponding set of simple roots. When dealing with an explicit irreducible root system we will enumerate the simple roots following Bourbaki's notation \cite{Bou}. We also set $\grD^- = -\grD$.

If $\grb \in \Phi^+$ and $\gra \in \grD$ we denote by $[\grb:\gra]$ the coefficient of $\gra$ in $\grb$ as a sum of simple roots, and we define the \textit{support} of $\grb$ as
$$
	\supp(\grb) = \{\gra \in \grD \st [\grb:\gra] > 0\}.
$$

Let $W = N_G(T)/T$ be the Weyl group of $G$ with respect to $T$. If $\gra \in \Phi$, we denote by $s_\gra \in W$ the corresponding reflection and by $U_\gra \subset G$ the unipotent root subgroup associated with $\gra$. If $\gra \in \grD$, we denote by $P_\gra$ the minimal parabolic subgroup of $G$ containing $B$ associated with $\gra$. If $w \in W$, we denote by $\Phi^+(w)$ the corresponding \textit{inversion set}, i.e.
$$
	\Phi^+(w) = \{\gra \in \Phi^+ \st w(\gra) \in \Phi^-\},
$$
and by $l(w)$ the length of $w$, that is the cardinality of $\Phi^+(w)$. Denote by $w_0$ be the longest element of $W$. If $\scrO$ is a $T$-stable subset of a $G$-variety and $n\in N_G(T)$, then $n\scrO$ only depends on the class $w$ of $n$ in $W$, therefore we will denote $n\scrO$ simply by $w\scrO$.

%%%%%%%%%%%%%%%%%%%%%%%%%%%%%%%%%%%%%%%%%%%%%%%
\subsection{Spherical varieties and toric varieties}	\label{ssec:spherical}
%%%%%%%%%%%%%%%%%%%%%%%%%%%%%%%%%%%%%%%%%%%%%%%

An irreducible normal $G$-variety $X$ is called \textit{spherical} if it contains an open $B$-orbit. See \cite{Kn1} as a general reference for spherical varieties. In particular, $X$ contains an open $G$-orbit, which is a spherical homogeneous variety. Following \cite{Br1} and \cite{Vi}, $X$ is spherical if and only if $B$ possesses finitely many orbits on it. Therefore, following the definition we gave in the introduction, a homogeneous variety $G/H$ is spherical if and only if $H$ is a spherical subgroup of $G$.

Let $X$ be a spherical $G$-variety. Every $B$-stable prime divisor $D \in \Div_B(X)$ induces a discrete valuation $\nu_D$ on $\mC(X)$, which is trivial on the constant functions. On the other hand, since $X$ contains an open $B$-orbit, every function $f \in \mC(X)^{(B)}$ is uniquely determined by its weight up to a scalar factor. By restricting valuations to $\mC(X)^{(B)}$ we get then a map
$$
	\rho : \Div_B(X) \lra \calX(X)^\vee,
$$
defined by $\langle \rho(D), \chi \rangle = \nu_D(f_\chi)$ where $f_\chi \in \mC(X)^{(B)}$ is any non-zero $B$-semiinvariant function of weight $\chi$.

When $G = B = T$ is a torus, we will also say that $X$ is a \textit{toric variety}. Notice that, differently form the standard literature on toric varieties, we will not assume that the action of $T$ on $X$ is effective. As a general reference on toric varieties see \cite{CLS}.

Let $X$ be an affine toric $T$-variety. The \textit{cone of} $X$ is the rational polyhedral cone in $\calX(X)^\vee_\mQ$ defined as
$$
	\grs_X = \cone(\rho(D) \st D \in \Div_T(X))
$$
This cone encodes important information on the geometry of $X$. In particular, there is an order reversing one-to-one correspondence between the set of $T$-orbits on $X$ (ordered with the inclusion of orbit closures) and the set of faces of $\grs_X$ (ordered with the inclusion). Given a face $\tau$ of $\grs_X$ we denote by $\scrU_\tau$ the corresponding $T$-orbit in $X$, in particular $\scrU_\grs$ is the unique closed orbit of $X$ and $\scrU_0$ is the open orbit of $X$.

\begin{definition}[{see \cite[Definition 2.3]{Ll}}]	\label{def:radici-toriche}
Let $X$ be an affine toric $T$-variety. An element $\alpha\in \calX(X)$ is called a \textit{root} of $X$ if there exists $\grd(\gra) \in \Div_T(X)$ such that $\langle \rho(\grd(\gra)), \alpha \rangle = -1$, and $\langle \rho(D), \alpha \rangle \geq 0$ for all $D \in \Div_T(X) \senza \{\grd(\gra)\}$.
\end{definition}

We will denote by $\Root(X)$ the set of roots of an affine toric variety $X$. Notice that by its definition $\Root(X)$ comes with a map
\begin{equation}	\label{eqn:mappa-delta-torica}
\grd : \Root(X) \lra \Div_T(X).
\end{equation}

For the notion of root of a toric variety, which goes back to Demazure, see also \cite[Definition 2.1]{AZK} and \cite[Proposition 3.13]{Oda}. Following \cite{De} and \cite{Ll}, there is a one-to-one correspondence between the roots of $X$ and the one parameter unipotent subgroups of $\Aut(X)$ normalized by $T$, see \cite[\S 2]{Ll} and \cite[\S 2]{AZK}. More precisely, every $\gra \in \Root(X)$ defines a locally nilpotent derivation $\partial_\gra$ of the graded algebra $\mC[X]$, which acts on $\mC[X]$ by the rule
\begin{equation}		\label{eqn:rooot-action}
\partial_\alpha(f_\chi) = \langle \rho (\grd(\gra)), \chi \rangle f_{\chi + \gra}
\end{equation}
where $\chi \in \calX(X)$ and where $f_\chi \in \mC(X)^{(T)}$ has weight $\chi$. Notice that the line $\mC \partial_\alpha$ is fixed by the action of $T$, which acts on it via the weight $\gra$. By exponentiating $\partial_\gra$, we get then a one parameter unipotent subgroup $V_\gra \subset \Aut(X)$ normalized by $T$, that is $V_\gra = \grl_\gra(\mathbb C)$ where $\grl_\gra$ denotes the one parameter subgroup $\xi \mapsto \exp(\xi \partial_\gra)$.

We gather in the following proposition some properties, that we will need later, of the action of the group $V_\gra$ on the $T$-orbits of $X$.

\begin{proposition} [{\cite[Lemma 2.1, Proposition 2.1 and Lemma 2.2]{AZK}}] \label{prop:sottovarietà toriche stabili}
Let $X$ be an affine toric $T$-variety, let $\tau, \tau'$ be faces of $\grs_X$ and let $\gra \in \Root(X)$.
\begin{itemize}
	\item[i)] $V_\gra \scrU_\tau$ decomposes into the union of at most two $T$-orbits.
	\item[ii)] $\grd(\gra)$ is the unique $T$-stable prime divisor of $X$ which is not stable under the action of $V_\gra$.
	\item[iii)] Suppose that $\scrU_{\tau'} \subset \ol{\scrU_\tau}$, then $\scrU_{\tau'} \subset V_\gra \scrU_{\tau}$ if and only if $\gra_{|\tau'} \leq 0$ and $\tau$ is the codimension one face of $\tau'$ defined by the equation $\tau = \tau' \cap \ker \gra$. 
\end{itemize}
\end{proposition}

%%%%%%%%%%%%%%%%%%%%%%%%%%%%%%%%%%%%
\section{Strongly solvable spherical subgroups and\\ associated toric varieties}\label{s:toricvar}
%%%%%%%%%%%%%%%%%%%%%%%%%%%%%%%%%%%%

From now on, if not differently stated, $H$ will be a strongly solvable spherical subgroup of $G$. Up to conjugating $H$ in $G$ we may assume that $H$ is contained in $B$, in which case $H^\mru = U \cap H$. Up to conjugating $H$ in $B$ we may and will also assume that $T_H = T \cap H$ is a maximal diagonalizable subgroup of $H$. By Lemma~\ref{lemma:B-rank} it follows then that $TH/H \simeq T/T_H$ is a closed $T$-orbit in $B/H$, and that $\calX_B(B/H) = \calX_T(T/T_H)$. In particular we get the equality $\rk B/H = \rk G - \rk H$.

In order to classify the strongly solvable spherical subgroups $H$ of $G$, Avdeev introduced the \textit{active roots} of $H$, defined as
$$\Psi = \{ \gra \in \Phi^+ \st U_\gra \not \subset H \},$$
and developed a combinatorial theory of such roots. We refer to \cite{Avd1}, \cite{Avd2} for details on this construction.

%%%%%%%%%%%%%%%%%%%%%%%%%%%%%%%%%%%%%%%%%%%%%%%
\subsection{The structure of $B/H$ as a toric variety}\label{s:BmodH}
%%%%%%%%%%%%%%%%%%%%%%%%%%%%%%%%%%%%%%%%%%%%%%%

Consider the projection $G/H \ra G/B$, denote by $\scrB^*(G/H)$ the set of the $B$-orbits in $G/H$ which project dominantly on $G/B$ and by $\Div_B^*(G/H)$ the set of the $B$-stable prime divisors of $G/H$ which project dominantly on $G/B$. Since $G/H$ possesses finitely many $B$-orbits, $\Div_B^*(G/H)$ equals the set of closures of the codimension one $B$-orbits in $\scrB^*(G/H)$.

Since it will be a fundamental object in what follows, we denote
$$\scrD = \Div_T(B/H).$$

\begin{proposition}	\label{prop:corrispondenza-divisori}
The map $\scrO \mapsto w_0 \scrO \cap B/H$ induces a bijection between $\scrB^*(G/H)$ and the set of $T$-orbits in $B/H$, which preserves codimensions and inclusions of orbit closures. In particular, the map $D \mapsto w_0 D \cap B/H$ induces a bijection between $\Div_B^*(G/H)$ and $\scrD$.
\end{proposition}

\begin{proof}
Let $\scrO \in \scrB^*(G/H)$. The image of $\scrO$ in $G/B$ is the dense $B$-orbit $Bw_0B/B$, therefore there exists $u \in B$ such that $\scrO = Bw_0 uH/H$. Then
$$w_0 \scrO \cap B/H = (B \cap B^{w_0}) u H/H = T uH/H,$$
that is, $w_0 \scrO \cap B/H$ is a single $T$-orbit. Conversely, if $\scrU \subset B/H$ is a $T$-orbit, then we get an element $\scrO \in \scrB^*(G/H)$ by setting $\scrO = Bw_0 \scrU$. The equalities $\scrO = Bw_0(w_0 \scrO \cap B/H)$ and $\scrU = w_0 (Bw_0\scrU) \cap B/H$ imply the first claim, and the rest is an obvious consequence.
\end{proof}

\begin{corollary}	\label{cor:torica}
$B/H$ is a smooth affine toric $T$-variety.
\end{corollary}

\begin{proof}
Since it is homogeneous under the action of a solvable group, $B/H$ is smooth and affine (see e.g. \cite[Lemma 2.12]{Ti1}). Notice also that $B/H$ is irreducible since $B$ is connected. Therefore the fact that $B/H$ is toric under the action of $T$ follows by the sphericity of $H$ thanks to Proposition~\ref{prop:corrispondenza-divisori}, since $\scrB^*(G/H)$ is a finite set.
\end{proof}

\begin{proposition}\label{prop:B/H-toric}
The following hold.
\begin{itemize}
	\item[i)] As a $T$-variety, the weight lattice of $B/H$ is $w_0\calX(G/H)$.
	\item[ii)] For any $D\in \Div_B^*(G/H)$, the $T$-invariant valuation of $B/H$ defined by the $T$-stable prime divisor $w_0 D\cap B/H \in \scrD$ coincides with $w_0 \rho(D)$.
\end{itemize}
\end{proposition}

\begin{proof}
Up to twisting the $T$-action by $w_0$, the $T$-varieties $B/H$ and $w_0B/H$ are isomorphic. Since the $B$-orbit of $w_0B/B$ is dense in $G/B$ and since $B\cap B^{w_0}=T$, the $T$-stable prime divisors of $w_0 B/H$ coincide with the intersections $D \cap w_0 B/H$, where $D \in \Div_B^*(G/H)$. To prove the proposition we may then replace $B/H$ with $w_0 B/H$, and show instead the equality $\calX_T(w_0 B/H) = \calX_B(G/H)$ and that the $T$-invariant valuation of $w_0B/H$ defined by $D\cap w_0B/H$ is $\rho(D)$.

Consider the open subset $Bw_0B/H \subset G/H$, and notice that we have an isomorphism
$Bw_0B/H \simeq B\times^T w_0 B/H$. Then restriction to $w_0 B/H$ induces a bijection between $B$-semiinvariant rational functions on $G/H$ and $T$-semiinvariant rational functions on $w_0 B/H$, and i) follows.

Identifying $\mC(G/H)^{(B)}$ and $\mC(w_0B/H)^{(T)}$, we get as well an identification between the discrete valuation of $\mC(G/H)$ associated with a $B$-stable prime divisor $D \subset G/H$ which intersects $Bw_0B/H$ (hence $w_0 B/H$) with the discrete valuation of $\mC(w_0 B/H)$ associated with $D\cap w_0 B/H$, and ii) follows.
\end{proof}

The set of active roots of $H$ is closely related to the structure of $B/H$ as a toric variety. Consider indeed the Levi decompositions $B =TU \simeq T \ltimes U$ and $H = T_HH^\mru \simeq T_H \ltimes H^\mru$. These induce a projection $B/H \ra T/T_H$, and a $T$-equivariant isomorphism
\begin{equation}		\label{eqn:fiber-product}
	B/H \simeq T\times^{T_H} U/H^\mru.
\end{equation}
In other words, $B/H$ is a homogeneous vector bundle over $T/T_H$, with fiber the $T_H$-module $U/H^\mru$. Therefore the $T$-orbits in $B/H$ correspond naturally to the $T_H$-orbits in $U/H^\mru$, and it follows that $U/H^\mru$ possesses an open $T_H$-orbit.

Denote $\tau : \calX(T) \ra \calX(T_H)$ the restriction map and set
$$
	\ol\Psi = \{\tau(\gra) \st \gra \in \Psi\}.
$$
Since $T_H$ is a diagonalizable group, the $T_H$-module structure of $U/H^\mru$ is completely determined by its $T_H$-weights, namely by $\ol\Psi$. The fact that $U/H^\mru$ possesses an open $T_H$-orbit is then equivalent to the fact that it is a {\em multiplicity-free} $T_H$-module, that is $\ol\Psi$ is linearly independent. Given $\pi \in \calX(T_H)$, we denote by $\mC_\pi$ the one dimensional $T_H$-module defined by $\pi$. By \cite[Lemma 1.4]{Mo} the exponential map induces a $T_H$-equivariant isomorphism $\gou/\goh^\mru \ra U/H^\mru$, hence we get isomorphisms of $T_H$-modules
\begin{equation}	\label{eq:U/H^u}
	U/H^\mru \simeq \gou/\goh^\mru \simeq \bigoplus_{\pi \in \ol\Psi} \mC_\pi.
\end{equation}

In particular, $\ol\Psi$ parametrizes the set $\Div_{T_H}(U/H^\mru)$, and summarizing the previous discussion we get canonical bijections
\begin{equation}		\label{eqn:identificazioni-divisori}
	\Div_B^*(G/H) \longleftrightarrow \scrD \longleftrightarrow \Div_{T_H}(U/H^\mru) \longleftrightarrow \ol\Psi. 
\end{equation}

\begin{definition}
A root $\alpha \in \Phi^+$ is called \textit{weakly active} for $H$ if $U_\gra$ acts non-trivially on $B/H$.
\end{definition}

We denote by $\weak \subset \Phi^+$ the set of weakly active roots of $H$. 

As an immediate consequence of the definitions, notice that every active root is weakly active. Moreover, if $\gra \in \weak$, by definition we have a non-trivial homomorphism $U_\gra \ra \Aut(B/H)$, which must be injective since $U_\gra$ has no non-trivial proper subgroups (see \cite[Theorem 20.5]{Hu}). On the other hand $U_\gra$ is normalized by $T$, which by definition acts on the Lie algebra $\gou_\gra$ with the character $\gra \in \calX(T)$. By the one-to-one correspondence between one parameter unipotent subgroups of $\Aut(B/H)$ normalized by $T$ and roots of $B/H$ recalled in Subsection \ref{ssec:spherical}, it follows that  every weakly active root for $H$ is a root for $B/H$ in the sense of Definition \ref{def:radici-toriche}, and in the notation therein we have an isomorphism $U_\gra \simeq V_\gra \subset \Aut(B/H)$. Therefore we have $\weak \subset \Root(B/H)$, and in particular $\weak \subset \calX_T(B/H)$.

Following Definition \ref{def:radici-toriche}, the inclusion $\weak \subset \Root(B/H)$ yields by restriction a map
$$
		\grd_{|\weak} : \weak \lra \scrD
$$

\begin{proposition} \label{prop:divisori-stabili}
Let $\gra \in \weak$. Then $\grd(\gra)$ is uniquely determined by $\tau(\gra)$, and it is the unique $T$-stable prime divisor of $B/H$ which is not $U_\gra$-stable. 
\end{proposition}

\begin{proof}
By \eqref{eqn:fiber-product}, the $T$-stable prime divisors of $B/H$ correspond to the $T_H$-stable prime divisors of $U/H^\mru$. If $D \in \scrD$, it follows that the evaluation of $\gra$ along $D$ coincides with the evaluation of its restriction $\tau(\gra) \in \calX_{T_H}(U/H^\mru)$ along the intersection $D \cap U/H^\mru \in \Div_{T_H}(U/H^\mru)$. In particular $\tau(\gra) \in \Root(U/H^\mru)$, and $\grd(\gra)$ is uniquely determined by the restriction $\tau(\gra)$.

As we already noticed, the morphism of algebraic groups $B \ra \Aut(B/H)$ maps $U_\gra$ onto the one parameter unipotent subgroup $V_\gra \subset \Aut(B/H)$ determined by $\gra$ as a root of $B/H$. Therefore the second claim follows by Proposition~\ref{prop:sottovarietà toriche stabili}.
\end{proof}

We will need the following characterization of the active roots in terms of the corresponding evaluation.

\begin{theorem} [{\cite[Theorem 5.34 (b)]{Avd2}}]  \label{teo:equivalent}
Let $\alpha\in \mN\Phi^+$ and assume that $\gra \in \mZ \Psi$. Then $\gra \in \Psi$ if and only if there exists $D_\gra \in \scrD$ such that $\langle \rho(D_\gra), \alpha \rangle = -1$ and $\langle \rho(D), \alpha \rangle = 0$ for all $D \in \scrD \senza \{D_\gra\}$, in which case $D_\gra = \grd(\gra)$.
\end{theorem}

For $\gra \in \weak$, let $f_{-\gra} \in \mC(B/H)^{(T)}$ be an eigenfunction of weight $-\gra$ (uniquely determined up to a scalar factor). Since $\gra \in \Root(B/H)$, it follows by Definition \ref{def:radici-toriche} that $f_{-\gra}$ vanishes with order $1$ along $\grd(\gra)$, and $\grd(\gra)$ is the unique $T$-stable prime divisor of $B/H$ where $f_{-\gra}$ vanishes. Notice that, when $\gra \in \Psi$, the function $f_{-\gra}$ is nothing but the obvious lifting of the coordinate of $\gou/\goh^\mru$ corresponding to $\tau(\gra)$ via \eqref{eqn:fiber-product} and \eqref{eq:U/H^u}.

%In terms of the $T$-equivariant isomorphism $B/H \simeq T \times^{T_H} \gou/\goh^\mru$ described in \eqref{eqn:fiber-product} and \eqref{eq:U/H^u}, notice that $f_{-\gra}$ is nothing but the obvious lifting of the coordinate of $\gou/\goh^\mru$ corresponding to $\tau(\gra)$. To summarize:

\begin{corollary}	\label{cor:weak-vs-active}
Let $\gra \in \weak$. Then $\gra \in \Psi$ if and only if $\langle D, \gra \rangle \leq 0$ for all $D \in \scrD$. If moreover $\gra \in \Psi$, then $f_{-\gra} \in \mC[B/H]$ is a global equation for $\grd(\gra)$, and $u \grd(\gra) \cap \grd(\gra)= \vuoto$ for all non trivial element $u \in U_\gra$.
\end{corollary}

\begin{proof}
The first implication follows by Theorem \ref{teo:equivalent}. Suppose that $\gra \in \weak$ and suppose that $\langle \rho (D), \gra \rangle \leq 0$ for all $D \in \scrD$. It follows that $f_{-\gra}$ has no pole on any $T$-stable divisor of the smooth variety $B/H$, hence $f_{-\gra} \in \mC[B/H]$. Moreover $f_{-\gra}$ is a global equation for $\grd(\gra)$, because it vanishes with order $1$ on $\grd(\gra)$, and it is non-zero on every other $T$-stable prime divisor of $B/H$.

We claim that $u \grd(\gra) \cap \grd(\gra) = \vuoto$ for all non-trivial elements $u \in U_\gra$. Let indeed $\partial_\gra$ be the locally nilpotent derivation of $\mC[B/H]$ associated to $\gra$. By (\ref{eqn:rooot-action}) it follows that $\partial_\alpha(f_{-\gra})$ is a non-zero function, which is constant on $B/H$. Exponentiating $\partial_\alpha$ we obtain $f_{-\alpha}(u_\gra(\xi)x) = f_{-\alpha}(x)+\xi$ for all $\xi \in \mC$ (where $\mC$ is regarded as the $T$-module of weight $\gra$ and $u_\gra : \mC \ra U_\alpha$ is a $T$-equivariant parametrization). Therefore $u(\xi) \grd(\gra) \cap \grd(\gra) = \vuoto$ for all $\xi \in \mC$ different from zero. 

Since $TH/H \subset B/H$ is the unique closed $T$-orbit, it holds $TH/H \subset \grd(\gra)$. Therefore by the previous discussion we get $U_\gra T H/H \not \subset T H/H$, namely $U_\gra \not \subset H$.
\end{proof}

\begin{corollary}
Let $\alpha,\beta\in \Psi$, then $\grd(\gra) = \grd(\grb)$ if and only if $\tau(\gra) = \tau(\grb)$. 
\end{corollary}

\begin{proof}
We already noticed in Proposition \ref{prop:divisori-stabili} that for all $\gra \in \weak$ the divisor $\grd(\gra)$ is uniquely determined by the restriction $\tau(\gra)$. Conversely, if $\gra, \grb \in \Psi$ and $\grd(\gra) = \grd(\grb)$, then by Corollary \ref{cor:weak-vs-active} it follows that the restrictions $(f_{-\gra})_{| U/H^\mru}$ and $(f_{-\grb})_{| U/H^\mru}$ both belong to $\mC[U/H^\mru]^{(T_H)}$ and they are both global equation for the $T_H$-stable prime divisor $\grd(\gra) \cap U/H^\mru = \grd(\grb) \cap U/H^\mru$. 

On the other hand by \eqref{eq:U/H^u} $U/H^\mru$ is $T_H$-equivariantly isomorphic to the toric module $\bigoplus_{\pi \in \ol\Psi} \mC_\pi$, and the $T_H$-stable prime divisors of $U/H^\mru$ are precisely the coordinate hyperplanes of $\bigoplus_{\pi \in \ol\Psi} \mC_\pi$. Therefore under the isomorphism \eqref{eq:U/H^u} the restrictions of $f_{-\gra}$ and $f_{-\grb}$ both correspond to the unique coordinate corresponding to $\grd(\gra) \cap U/H^\mru = \grd(\grb) \cap U/H^\mru$. It follows that $(f_{-\gra})_{| U/H^\mru}$ and $(f_{-\grb})_{| U/H^\mru}$ have the same $T_H$-weight, namely $\tau(\gra) = \tau(\grb)$.
\end{proof}

\begin{corollary}	\label{cor:delta-surjective}
Via the bijection $\scrD \leftrightarrow \ol\Psi$ of (\ref{eqn:identificazioni-divisori}), the map $\grd_{|\Psi} : \Psi \ra \scrD$ is identified with $\tau_{|\Psi}: \Psi \ra \ol\Psi$. In particular, $\grd_{|\Psi}$ is surjective.
\end{corollary}

%%%%%%%%%%%%%%%%%%%%%%%%%%%%%%%%%%%%%%%%%%%%%%%
\subsection{Orbits of $T$ on $B/H$ via active roots}
%%%%%%%%%%%%%%%%%%%%%%%%%%%%%%%%%%%%%%%%%%%%%%%

Since it is an affine toric variety, $B/H$ possesses a unique closed $T$-orbit, namely $TH/H$, which is contained in every $T$-stable divisor of $B/H$. Moreover, since it is smooth, the cone associated to $B/H$ is simplicial, therefore the $T$-orbits in $B/H$ are parametrized by the subsets of $\scrD$. Given $I \subset \scrD$ we denote by $\scrU_I$ the corresponding $T$-orbit in $B/H$. To be more explicit, by Proposition \ref{prop:B/H-toric} $\scrU_I$ is defined by the equality
$$
	I = \{D \in \scrD \st \scrU_I \not \subset D\}.
$$
In particular, $\scrU_\vuoto$ is the closed orbit of $B/H$, and $\scrU_\scrD$ is the open orbit. The goal of this subsection is to give an explicit description of the $T$-orbits $\scrU_I$, by giving canonical base points defined in terms of the active roots of $H$.

\begin{proposition}	\label{prop:Torb_onD}
Fix an enumeration $\scrD = \{D_1, \ldots, D_m\}$, let $1 \leq i_1 < \ldots < i_p \leq m$ and set $I = \{D_{i_1}, \ldots, D_{i_p}\}$. For every choice of elements $\grb_1, \ldots, \grb_p \in \Psi$ with $\grd(\grb_1) = D_{i_1}, \ldots, \grd(\grb_p) = D_{i_p}$ and for every choice of non-trivial elements $u_1 \in U_{\grb_1}, \ldots, u_p \in U_{\grb_p}$, it holds the equality
$$
	\scrU_I = T u_1 \cdots u_p H/H.
$$
\end{proposition}

\begin{proof}
Denote $\scrU = T u_1 \cdots u_p H/H$. The point $u_1 \cdots u_p H/H$ is obtained by successively applying the elements $u_j$ on the base point $H \in B/H$, which is contained in every $T$-stable prime divisor of $B/H$. By Proposition~\ref{prop:divisori-stabili} the root group $U_{\beta_j}$ stabilizes every $D \in \scrD \senza \{\delta(\beta_j)\}$. Therefore the definition of $\grb_1, \ldots, \grb_p$ implies that $\scrU_I$ is contained in every $D$, for all $D \in \scrD \senza \{ D_{i_1}, \ldots, D_{i_p} \}$.

It remains to prove that $\scrU_I$ is not contained in $\delta(\beta_j)$ for all $j= 1, \ldots, p$. By Proposition~\ref{prop:divisori-stabili} it follows that $\delta(\beta_j)$ is stable under the action of $U_{\beta_k}$ for all $k \neq j$, whereas by Corollary~\ref{cor:weak-vs-active} we have $u_j \delta(\beta_j) \cap \delta(\beta_j) = \vuoto$. It follows that $u_{j+1} \cdots u_p H/H \in \delta(\beta_j)$ and $u_j \cdots u_p H/H \not \in \delta(\beta_j)$. Therefore $u_1 \cdots u_p H/H \not \in \delta(\beta_j)$ as well, and the claim follows.
\end{proof}

We deduce immediately from Proposition~\ref{prop:Torb_onD} the following description of the $T$-stable prime divisors of $B/H$ in terms of active roots.

\begin{corollary} \label{cor:divisors}
With the notation of Proposition~\ref{prop:Torb_onD}, we have
$$
	\grd(\grb_i) = \ol{T U_{\grb_1} \cdots \hat{U}_{\grb_i} \cdots U_{\grb_m} H/H}
$$
\end{corollary}

%%%%%%%%%%%%%%%%%%%%%%%%%%%%%%%%%%%%%%%%%%%%%%%%%%%
%%%%%%%%%%%%%%%%%%%%%%%%%%%%%%%%%%%%%%%%%%%%%%%%%%%
\section{Weakly active roots and root systems}\label{s:active}
%%%%%%%%%%%%%%%%%%%%%%%%%%%%%%%%%%%%%%%%%%%%%%%%%%%
%%%%%%%%%%%%%%%%%%%%%%%%%%%%%%%%%%%%%%%%%%%%%%%%%%%

In this section we will extend the study of the active roots to the weakly active roots, and we will give two alternative combinatorial definitions of them. This will provide the technical tools needed for the description of $\scrB(G/H)$ given in the next section. We keep the assumptions of the previous section.

Recall the \textit{dominance order} on $\calX(T)$, defined by $\grl \leq \mu$ if and only if $\mu - \grl \in \mN\grD$. Given $\gra \in \Psi$, denote $F(\gra)$ the \textit{family of active roots generated by} $\gra$, defined as
$$
	F(\gra) = \{ \grb \in \Psi \st \grb \leq \gra\}.
$$
We collect in the following proposition some properties of the families of active roots from \cite{Avd1} that will be very useful in what follows.

\begin{proposition}[{\cite[Lemma 5, Corollaries 2 and 3]{Avd1}}]	\label{prop:family}
Let $\gra\in \Psi$.
\begin{itemize}
	\item[i)] The family $F(\gra)$ is a linearly independent set of roots, and $\grd(\grb) \neq \grd(\grb')$ for all $\grb, \grb' \in F(\gra)$.
	\item[ii)] If $\grb \in \Psi$, then $\grb \in F(\gra)$ if and only if $\supp(\grb) \subset \supp(\gra)$.
\end{itemize}
\end{proposition}

%%%%%%%%%%%%%%%%%%%%%%%%%%%%%%%%%%%%%%%%%%%%%%%
\subsection{Weakly active roots}
%%%%%%%%%%%%%%%%%%%%%%%%%%%%%%%%%%%%%%%%%%%%%%%

We will give two different combinatorial characterizations of the weakly active roots. To do this, we will need a couple of preliminary lemmas.

\begin{lemma}	\label{lemma:somma-vive1}
Let $\gra \in \weak$ and $\grb \in \Psi$ be such that $\langle \rho(\grd(\grb)),\gra \rangle > 0$. Then $\gra + \grb \in \weak$, and $\grd(\gra+\grb) = \grd(\gra)$.
\end{lemma}

\begin{proof}
Notice that $\gra+\grb \in \Root(B/H)$: indeed $\gra + \grb \in \calX(B/H)$ and by Theorem~\ref{teo:equivalent} it satisfies the condition of Definition \ref{def:radici-toriche}. Moreover, as $\langle \rho(\delta(\grb)), \gra +\grb \rangle \geq 0$, it follows that $\grd(\gra+\grb) = \grd(\gra)$.

To conclude the proof, we need to show that $\gra +\grb \in \Phi^+$ and that $U_{\gra+\grb}$ acts non-trivially on $B/H$. As in Section~\ref{ssec:spherical}, given $\grg \in \Root(B/H)$, we denote by $V_\grg$ the associated one parameter unipotent subgroup of $\Aut(B/H)$, so that $\Lie(V_\gamma)$ acts as a derivation on $\mC(B/H)$ according to (\ref{eqn:rooot-action}). In particular, given $\grg_1, \grg_2 \in \Root(B/H)$, we get that $\partial_{\grg_1}$ and $\partial_{\grg_2}$ commute if and only if
$$
\langle \rho (\grd(\grg_1)), \grg_2 \rangle = \langle \rho (\grd(\grg_2)), \grg_1 \rangle.
$$

Since $\langle \rho(\delta(\gra)), \grb \rangle = 0$ and $\langle \rho (\grd(\grb)), \gra \rangle > 0$, it follows by the previous discussion that $\partial_\gra$ and $\partial_\grb$ do not commute, hence $\Lie(V_{\gra+\grb}) = [\Lie(V_\gra),\Lie(V_\grb)]$. Consider now the homomorphism of algebraic groups $\psi : B \ra \Aut(B/H)$ induced by the action of $B$ on $B/H$. By definition, $U_\gra$ and $U_\grb$ act non-trivially on $B/H$, hence $\psi(U_\gra) = V_\gra$ and $\psi(U_\grb) = V_\grb$. Therefore
$$
	\Lie(V_{\gra+\grb}) = [\Lie(V_\gra),\Lie(V_\grb)] = d\psi [\gou_\gra,\gou_\grb].
$$
It follows that $[\gou_\gra,\gou_\grb]$ is non-zero, hence $\gra + \grb \in \Phi^+$ and $\psi(U_{\gra+\grb}) = V_{\gra+\grb}$.
\end{proof}

Given a root $\gra \in \Phi^+$, we define the following sets:
$$\Phi^+(\alpha) = (\alpha + \mN \Phi^+) \cap \Phi^+, \qquad \qquad
\Psi(\gra) = (\gra + \mN\Psi) \cap \Psi.$$

\begin{lemma}	\label{lemma:somma-radici}
Let $\alpha\in\Phi^+$ and $\grb \in \Phi^+(\alpha) \senza \{\alpha\}$. There exist $\grb_0, \ldots, \grb_n \in \Phi^+$ with $\grb_0 = \gra$, $\grb_n = \grb$ and $\grb_i - \grb_{i-1} \in \Phi^+$ for all $i \leq n$. 
\end{lemma}

\begin{proof}
Let $\gra_1, \ldots, \gra_m \in \Phi^+$ with $\grb = \gra + \gra_1 + \ldots + \gra_m$. Assume that $m$ is minimal with this property, we prove the claim by induction on $m$.
Set $\grg = \beta - \alpha$. If $m=1$ the claim is true. Assume $m>1$, then by the minimality of $m$ it follows that $\grg \not \in \Phi^+$. Since $\grg \in \mN \Phi^+$, it follows then $\grg \not \in \Phi$, hence $(\grb,\alpha) \leq 0$, where $(-,-)$ denotes an ad-invariant scalar product on $\got^*$. Hence
\[
(\alpha,\alpha) + (\grg,\alpha) = (\alpha + \grg, \alpha) = (\grb, \alpha) \leq 0
\]
an we get $(\grg, \alpha) \leq -(\alpha,\alpha) < 0$. Up to reordering the indices, we may assume that $(\gra_1,\alpha) < 0$, hence $\gra' = \gra + \gra_1 \in \Phi^+$. On the other hand $\grb \in \Phi^+(\gra')$ and $\grb = \gra' + \gra_2 + \ldots + \gra_m$, therefore the claim follows by induction.
\end{proof}

\begin{lemma}	\label{lemma:ideali}
Given $\alpha \in \Phi^+$, the ideal $(\gou_\alpha)$ generated by $\gou_\alpha$ in $\gou$ is the direct sum of the root spaces $\gou_\grb$ with $\grb \in \Phi^+(\alpha)$.
\end{lemma}

\begin{proof}
Denote $\gor = \bigoplus_{\grb \in \Phi^+(\alpha)} \gou_\grb$. We show that $\gor \subset \gou$ is an ideal and that $\gor \subset (\gou_\alpha)$, whence the lemma. To show that $\gor$ is an ideal, it is enough to notice that for all $\grb \in \Phi^+(\alpha)$ and for all $\gamma \in\Phi^+$ it holds either $[\gou_\grb, \gou_\grg] = 0$ or $[\gou_\grb, \gou_\grg] = \gou_{\grb + \grg}$, in which case $\grb + \grg \in \Phi^+(\alpha)$.

Let $\grb \in \Phi^+(\alpha)$. By Lemma~\ref{lemma:somma-radici}, there exist $\grb_0, \ldots, \grb_n \in \Phi^+$ with $\grb_0 = \gra$, $\grb_n = \grb$ and $\grb_i - \grb_{i-1} \in \Phi^+$ for all $i \leq n$. For every $i=1, \ldots, n-1$ define $\gra_i = \grb_i - \grb_{i-1}$, then we have
\[
[\gou_{\beta_i},\gou_{\gra_{i+1}}] = \gou_{\beta_{i+1}} \neq 0.
\]
It follows that $\gou_\grb \subset (\gou_\alpha)$ for all $\grb \in \Phi^+(\gra)$, hence $\gor \subset (\gou_\gra)$.
\end{proof}

\begin{theorem}\label{teo:deadker}
Let $\gra \in \Phi^+$. The following statements are equivalent:
\begin{itemize}
	\item[i)] $\gra \in \weak$;
	\item[ii)] $\Psi(\gra) \neq \vuoto$;
	\item[iii)] $\Phi^+(\gra) \cap \Psi \neq \vuoto$.
\end{itemize}
%$$
%	\gra \in \weak \iff \Phi^+(\gra) \cap \Psi \neq \vuoto \iff \Psi(\gra) \cap \Psi \neq \vuoto \iff U_\gra \text{ acts non-trivially on } B/H.
%$$
If moreover $\gra \in \weak$, then we have $\delta(\gra) = \delta(\grb)$ for all $\beta \in \Psi(\gra)$.
\end{theorem}

\begin{proof}
%i) $\Rightarrow$ ii) This follows from the correspondence between roots of $B/H$ and one-parameter unipotent subgroups of $\Aut(B/H)$ that we discussed after Definition~\ref{def:radici-toriche}.

%We don't conclude the equivalence of i) and ii) yet, although $\Aut(B/H)$ does contain the unipotent subgroup $V_\alpha$ if $\alpha\in\weak$, because the map $B\to \Aut(B/H)$ induced by the action might nevertheless send $U_\alpha$ to $\{ \mathrm{id}_{B/H}\}$.

i)  $\Rightarrow$ ii) For all $D \in \scrD$, fix a root $\beta_D \in \Psi$ such that $\delta(\beta_D) = D$. This is possible thanks to Corollary~\ref{cor:delta-surjective}. Let now $\gra \in \weak$ and consider the element
\begin{equation}\label{eqn:activate}
\grb = \alpha + \sum_{D\in\scrD \senza \{\grd(\gra)\}} \langle \rho(D), \gra \rangle \grb_D.
\end{equation}
By Lemma \ref{lemma:somma-vive1} it follows that $\grb \in \weak$. On the other hand, if $D \in \scrD$, then Proposition \ref{prop:B/H-toric} and Theorem~\ref{teo:equivalent} imply
$$
	\langle \rho(D), \grb \rangle =
	\left\{ \begin{array}{ll}
			0 & \text{ if } D \neq \grd(\gra) \\
			-1 & \text{ if } D = \grd(\gra)
	\end{array}\right.
$$
Therefore $\grb \in \Psi$ by Corollary \ref{cor:weak-vs-active}, and we get ii).

ii) $\Rightarrow$ iii) Obvious.

iii) $\Rightarrow$ i) Suppose that $U_\gra$ acts trivially on $B/H$, we show that $\Phi^+(\gra) \cap \Psi = \vuoto$. Denote $N$ the kernel of the action of $U$ on $B/H$. Given $u \in U$, we have $u \in N$ if and only if $ubH = bH$ for all $b\in B$. On the other hand, for all $b\in B$, the equality $ubH = bH$ holds if and only if $u\in bHb^{-1}$. Therefore $N$ is the intersection of $U$ with the biggest normal subgroup of $B$ contained in $H$. Equivalently, $N$ is the biggest $T$-stable normal subgroup of $U$ contained in $H$. It follows that $N$ is stable under conjugation by $T$, hence it is the product of the root subgroups $U_\alpha$ which are contained in it.

By assumption we have $U_\gra \subset N$, hence $\gou_\gra \subset \gon$. By Lemma~\ref{lemma:ideali} the ideal $(\gou_\gra)$ generated by $\gou_\gra$ in $\gou$ is the direct sum of the root spaces $\gou_\grb$ with $\grb \in \Phi^+(\gra)$. On the other hand $\gon$ is an ideal of $\gou$, and by definition no root space $\gou_\gra$ with $\gra \in \Psi$ is contained in $\goh$. Therefore $\Phi^+(\gra) \cap \Psi = \vuoto$.

For the last statement of the theorem, let $\gra \in \weak$ and $\beta \in \Psi(\alpha)$. By Theorem~\ref{teo:equivalent} it follows that $\langle \rho(\grd(\gra)), \grb \rangle \leq -1$. On the other hand $\langle \rho(D), \grb \rangle = 0$ for all $D \in \scrD \senza \{\grd(\grb)\}$, therefore $\delta(\beta) = \delta(\alpha)$.
\end{proof}

As a first consequence of the previous theorem, we generalize Lemma \ref{lemma:somma-vive1} to arbitrary pairs of weakly active roots.

\begin{lemma}	\label{lemma:somma-vive}
Let $\gra, \grb  \in \weak$ be such that $\langle \rho(\grd(\grb)),\gra \rangle > 0$. Then $\gra + \grb \in \weak$, and $\grd(\gra+\grb) = \grd(\gra)$.
\end{lemma}

\begin{proof}
Let $\grb' \in \Psi(\grb)$, then $\grd(\grb') = \grd(\grb)$ and Lemma \ref{lemma:somma-vive1} implies that $\gra + \grb' \in \weak$, and $\grd(\gra+\grb') = \grd(\gra)$. Let $\grg \in \Psi(\gra+\grb')$, then $\grg = \gra + \grb + \grg_1+ \ldots+\grg_n$ for some $\grg_1, \ldots, \grg_n \in \Psi$ and Theorem \ref{teo:deadker} implies $\grd(\grg) = \grd(\gra)$. Therefore, by Theorem \ref{teo:equivalent}, $\langle \rho(D), \grg \rangle = 0$ for all $D \in \scrD\senza \{\grd(\gra)\}$, and $\langle \rho(\grd(\gra)), \grg \rangle = -1$. 
As $\{\grg_1, \ldots, \grg_n\} \subset F(\grg)$, by Proposition~\ref{prop:family} i) it follows $\grd(\grg_i) \neq \grd(\grg)$ for all $i$, hence $\langle \rho(\grd(\gra)), \grg_i \rangle = 0$ for all $i$ by Theorem \ref{teo:equivalent}, and it follows
$$\langle \rho(\grd(\gra)), \gra+\grb \rangle = \langle \rho(\grd(\gra)), \grg \rangle = -1.$$
On the other hand the hypothesis of the lemma implies $\langle \rho(\grd(\grb)), \gra + \grb \rangle \geq 0$, hence $\langle \rho(D), \gra + \grb \rangle \geq 0$ for all $D \in \scrD \senza \{\grd(\gra)\}$ and following Definition \ref{def:radici-toriche} we get $\gra + \grb \in \Root(B/H)$ and $\grd(\gra+\grb) = \grd(\gra)$.

To conclude the proof, we need to show that $\gra +\grb \in \Phi^+$ and that $U_{\gra+\grb}$ acts non-trivially on $B/H$. This is shown with the same argument used in Lemma \ref{lemma:somma-vive1}, which applies without any change in this more general case.
\end{proof}

%For all spherical root $\sigma$, the element $-w_0(\sigma)$ is weakly active. This follows from \cite[Theorem 5.28]{Avd2}, which states that the set of spherical roots $-w_0(\Sigma)$ of the homogeneous space $G/{}^{w_0}H$ coincides with the union of the supports of all positive roots $\alpha$ such that $U_{-\alpha}\notin {}^{w_0}H$.

Given $I\subset \scrD$ we denote
$$
	\Psi_I = \{\gra \in \Psi \st \grd(\gra) \in I\}
$$
In case $I = \{D\}$ is a single element, then we will also denote $\Psi_I$ simply by $\Psi_D$.

\begin{definition}	\label{def:activated-roots}
Let $I \subset \scrD$ and $\gra \in \weak$, then we say that $\gra$ is {\em activated by} $I$ if there is $\grb \in \mN\Psi_I$ with $\gra + \grb \in \Psi$, and we say that $\gra$ \textit{stabilizes} $I$ if there is $\grb \in \mN\Psi_I$ with $\gra + \grb \in \Psi_I$.
\end{definition}

We denote by
$$
	\Theta_I =  \{\gra \in \weak \, \st \,  \exists \grb \in \mN \Psi_I \text{ with } \gra + \grb \in \Psi \}
$$
the set of the weakly active roots activated by $I$, and by $\weak_I$ the set of the weakly active roots which stabilize $I$. By Theorem \ref{teo:deadker} we have 
$$
	\weak_I = \{\gra \in \Theta_I \, \st \,  \grd(\gra) \in I \}.
$$
Notice that, for all $I \subset \scrD$, we have $\Psi \subset \Theta_I$ and $\Psi_I \subset \weak_I$.

In the language just introduced, we have the following consequence of Theorem~\ref{teo:equivalent}.

\begin{proposition} \label{prop:radici-stabili}
Let $\gra \in \weak$ and $I \subset \scrD$.  Then
\begin{itemize}
	\item[i)] $\gra \in \Theta_I$ if and only if $\langle \rho(D), \gra \rangle \leq 0$ for all $D \in \scrD \senza I$,
	\item[ii)] $\gra \in \weak_I$ if and only if $\langle \rho(D), \gra \rangle = 0$ for all $D \in \scrD \senza I$.
\end{itemize}
In particular, we have the equality $\weak_I = \mN \weak_I \cap \weak$.
\end{proposition}

\begin{proof}
Let $\gra \in \Theta_I$, then there exists $\grb \in \Psi$ such that $\grb-\gra \in \mN \Psi_I$. Applying Theorem~\ref{teo:equivalent} to $\grb$ it follows that $\langle \rho(D), \gra \rangle \leq 0$ for all $D \in \scrD \senza I$. Conversely, if such inequalities are all satisfied, then we may consider the element $\grb \in \Psi(\gra)$ defined as in (\ref{eqn:activate}). Then by Theorem \ref{teo:deadker} we have $\grd(\grb) = \grd(\gra)$, and since $\langle \rho(D), \gra \rangle \geq 0$ for all $D \in \scrD \senza \{\grd(\gra)\}$ it follows that $\grb - \gra \in \mN \Psi_I$. This shows i). Claim ii) follows by noticing that, if $\gra \in \Theta_I$, then by definition $\gra \in \weak_I$ if and only if $\delta(\alpha) \in I$, if and only if $\langle \rho(D), \grb \rangle \geq 0$ for all $D \in \scrD \senza I$.
\end{proof}

Given $I \subset \scrD$, let $\grs_I \subset \calX(B/H)^\vee_\mQ$ be the cone of the affine open subset of $B/H$ with closed $T$-orbit $\scrU_I$. By definition $\grs_I$ is generated by the one dimensional rays in $\calX(B/H)^\vee_\mQ$ which correspond to the $T$-stable prime divisors of $B/H$ which contain $\scrU_I$, namely the elements of $\scrD \senza I$. Hence
$$
	\grs_I = \cone(\rho(D) \st D \in \scrD \senza I),
$$
and we may rephrase Proposition~\ref{prop:radici-stabili} as follows.

\begin{proposition} \label{prop:radici-stabili2}
Let $\gra \in \weak$ and $I \subset \scrD$. Then $\gra \in \Theta_I$ if and only if $\gra_{|\grs_I} \leq 0$, and $\gra \in \weak_I$ if and only if $\gra_{|\grs_I} = 0$.
\end{proposition}

Given $\gra \in \weak$, we are now ready to describe combinatorially the action of the root subgroup $U_\gra$ on the set of $T$-orbits on $B/H$ in terms of weakly active roots.

\begin{proposition}\label{prop:Ualphaactionfinal}
Let $\gra \in \weak$ and $I \subset \scrD$. Then $\scrU_I$ is not stable under the action of $U_\gra$ if and only if $\gra \in \Theta_I$. If this is the case, then we have
$$
	U_\alpha \scrU_I  = \left\{ \begin{array}{ll}
	\scrU_I \cup \scrU_{I \cup \{\delta(\gra)\}} & \text{if $\gra \in \Theta_I \senza \weak_I$}	\\
	\scrU_I \cup \scrU_{I \senza \{\delta(\gra)\}} & \text{if $\gra \in \weak_I$}
	\end{array} \right.
$$
In particular, $U_\gra \scrU_I \subset \ol \scrU_I$ if and only if either $\gra \in \weak_I$ or $\gra \not \in \Theta_I$.
\end{proposition}

\begin{proof}
Consider the homomorphism of algebraic groups $B \ra \Aut(B/H)$, then in the notation of Section~\ref{ssec:spherical} the root group $U_\gra$ acts on $B/H$ as the one parameter unipotent subgroup $V_\gra \subset \Aut(B/H)$. Suppose that $\scrU_I$ is not stable under the action of $U_\alpha$. Then by Proposition~\ref{prop:sottovarietà toriche stabili} i) the set $U_\gra \scrU_I$ decomposes in the union of two $T$-orbits $\scrU_I \cup \scrU_J$ for some $J$.

Suppose $|I| < |J|$, then $\scrU_I \subset \ol{\scrU_J}$ and Proposition~\ref{prop:sottovarietà toriche stabili} iii) implies $\gra_{|\grs_I} \leq 0$ and $\grs_J = \grs_I \cap \ker \gra$ is a facet of $\grs_I$. By Proposition~\ref{prop:radici-stabili2} we get then $\gra \in \Theta_I \senza \weak_I$. If instead $|I| > |J|$, then $\scrU_J \subset \ol{\scrU_I}$ and the same proposition implies $\gra_{|\grs_J} \leq 0$ and $\grs_I = \grs_J \cap \ker \gra$, hence $\gra \in \weak_I$ by Proposition~\ref{prop:radici-stabili2}.

Conversely, suppose that $\gra \in \Theta_I$. Then Proposition~\ref{prop:radici-stabili2} implies $\gra_{|\grs_I} \leq 0$, and by the description of $\grs_I$ it follows $\gra_{|\grs_I} = 0$ if and only if $\grd(\gra) \in I$. Then the claim follows by Proposition~\ref{prop:sottovarietà toriche stabili} iii) applied to the pair $\scrU_{I \senza \grd(\gra)} \subset \ol{\scrU_I}$ or to the pair $\scrU_I \subset \ol{\scrU_{I \cup \grd(\gra)}}$, depending on whether $\grd(\gra) \in I$ or $\grd(\gra) \not \in I$.
\end{proof}

\begin{remark}
In particular, given $\gra \in \weak$ and $I \subset \scrD$, we have the following properties:
\begin{itemize}
	\item[i)] If $U_\gra \scrU_I \neq \scrU_I$, then the closure $\ol{\scrU_I}$ is $U_\gra$ stable if and only if $\gra \in \weak_I$.
	\item[ii)] If $\gra \in \weak_I$, then $\gra \in \Theta_{I \senza \{\delta(\gra)\}}$.
\end{itemize}
\end{remark}

%%%%%%%%%%%%%%%%%%%%%%%%%%%%%%%%%%%%%%%%%%%%%%%
\subsection{Combinatorics related to the weakly active roots.}

Recall the following fundamental property of the active roots.

\begin{lemma}[{\cite[Lemma 7]{Avd1}}]	\label{lemma:differenza-attive-avdeev}
For all $\grb \in F(\gra)\senza\{\gra\}$ it holds $\gra - \grb \in \Phi^+$.
\end{lemma}

We will need a generalization of the previous property, which holds for any combination of roots in $F(\gra)$.

\begin{lemma} \label{lemma:differenza-attive}
Let $\gra \in \Psi$ and let $\grb \in \mN \Psi$ be such that $\grb < \gra$, then $\gra - \grb \in \Phi^+$.
\end{lemma}

\begin{proof}
Write $\gra-\grb = \gra_1 + \ldots + \gra_n$ with $\gra_i \in \Phi^+$, we show the claim by induction on $n$. We have $\alpha\in\Phi^+(\alpha_i)\cap\Psi$ for all $i\in\{1,\ldots,n\}$, so Theorem~\ref{teo:deadker} yields $\alpha_i\in\weak$.

Notice that by Theorem \ref{teo:equivalent} we have $\langle \rho(D),\beta \rangle \leq 0$ for all $D \in \scrD$, hence $\langle \rho(D),\alpha-\beta \rangle < 0$ at most for one $D \in \scrD$, and  if such $D$ exists then it is $\grd(\gra)$. On the other hand, write $\grb = \grb_1 + \ldots + \grb_m$ for $\grb_1,\ldots,\grb_m \in \Psi$: then $\grb_j\leq\gra$ for all $j$ and from Proposition~\ref{prop:family} i) we deduce $\grd(\grb_j) \neq \grd(\gra)$. Therefore $\langle \rho(\grd(\gra)),\gra-\grb \rangle =-1$, and following Definition \ref{def:radici-toriche} we get $\gra_1+\ldots+\gra_n \in \Root(B/H)$.

Suppose that $n > 1$, then there is at least one index $i$ with $\grd(\gra_i) \neq \grd(\gra)$, say $i=1$: this implies $\langle \rho(\grd(\gra_1)), \gra \rangle = 0$ by Theorem \ref{teo:equivalent}. Since $\langle \rho(\grd(\gra_1)), \gra_1 \rangle = -1$ and $\langle \rho(\grd(\gra_1)), \grb \rangle \leq 0$, it follows that there is at least one index $i$ such that $\langle \rho(\grd(\gra_1)), \gra_i \rangle > 0$, say $i=2$. It follows by Lemma \ref{lemma:somma-vive} that $\gra_1+\gra_2 \in \weak$, therefore we can apply the inductive hypothesis and we get $\gra \in \Phi^+$.
\end{proof}

Thanks to Theorem~\ref{teo:deadker}, we deduce the following descriptions.

\begin{corollary}	\label{cor:weak-as-differences}
Let $I \subset \scrD$, then we have the equalities
\[
	\Theta_I = \{\gra - \grb \st \gra \in \Psi, \; \grb \in \mN \Psi_I, \; \grb < \gra\},	
\]
\[
	\weak_I = \{\gra - \grb \st \gra \in \Psi_I, \; \grb \in \mN \Psi_I, \; \grb <  \gra\}.
\]
\end{corollary}

Let $\gra \in \weak$ and let $\grb \in \Psi$ be such that $\gra + \grb \in \Psi$, then \cite[Proposition 1]{Avd1} shows that $\gra + \grb' \in \Psi$ for all $\grb' \in \Psi$ with $\grd(\grb') = \grd(\grb)$. We will need the following generalization of this property.

\begin{proposition} 	\label{prop:attivazioni-equivalenti}
Let $\gra \in \mN \Phi^+$ and let $\grb_1, \ldots, \grb_n \in \Psi$ be such that $\gra + \sum_i a_i \grb_i \in \Psi$ for some $a_1, \ldots, a_n \in \mN$. Then $\gra \in \weak$, and $\gra + \sum_i a_i \grb_i' \in \Psi$ for all $\grb'_1, \ldots, \grb'_n \in \Psi$ such that $\grd(\grb_1) = \grd(\grb_1'), \ldots, \grd(\grb_n) = \grd(\grb_n')$. Moreover we have $\grd(\gra + \sum_i a_i \grb_i) = \grd(\gra + \sum_i a_i \grb'_i)$.
\end{proposition}

\begin{proof}
Denote $\grg = \gra + \sum_i a_i \grb_i$. We show the claim by induction on the sum $a = \sum_i a_i$. Suppose that $a = 1$, then $n=1$ and Lemma~\ref{lemma:differenza-attive-avdeev} shows $\gra = \grg - \grb_1 \in \Phi^+$. Therefore the claim follows by \cite[Proposition 1]{Avd1}.

Suppose now that $n >1$, then $\grb_n \in F(\grg)$ and Lemma~\ref{lemma:differenza-attive-avdeev} implies
$$
	\gra + a_1 \grb_1 + \ldots + a_{n-1} \grb_{n-1} + (a_n-1)\grb_n \in \Phi^+.
$$
Setting $\gra' = \gra + \grb_n'$, by \cite[Proposition 1]{Avd1} it follows
$$
	\gra' + a_1 \grb_1 + \ldots + a_{n-1} \grb_{n-1} + (a_n-1)\grb_n \in \Psi,
$$
and we conclude that $\gra + \sum_i a_i \grb_i' \in \Psi$ by the inductive hypothesis.

For the last claim, notice that $\gra \in \weak$ by Lemma~\ref{lemma:differenza-attive}. Denote $\grg' = \gra + \sum a_i \grb'_i$, then we have $\grg, \grg' \in \Psi(\gra)$, therefore $\grd(\grg') = \grd(\grg) = \grd(\gra)$ by Theorem~\ref{teo:deadker}.
\end{proof}

We now show that the semigroup generated by the set of weakly active roots $\weak_I$ which stabilize a given subset $I \subset \scrD$ is \textit{saturated} in the root lattice $\mZ \grD$, that is
\[
	\mN \weak_I = \cone(\weak_I) \cap \mZ \grD.
\]
More precisely, we will show that $\mN\weak_I$ is the intersection of the semigroup generated by the positive roots with the rational vector space generated by $\Psi_I$. In order to do this, we will need a couple of preliminary lemmas.

\begin{definition}\label{def:spherical-roots}
We denote by $\grS$ the union of the sets $\supp(\grb)$ for $\grb \in \Psi$.
\end{definition} 

\begin{proposition}	\label{prop:spherical-roots}
We have the equality $\grS = \weak \cap \grD$. In particular $\mZ \Psi = \mZ \weak = \mZ \grS$ and $\mZ\Psi \cap \mN\grD = \mN \weak = \mN \grS$.
\end{proposition}

\begin{proof}
Let $\gra \in \weak \cap \grD$, then by Theorem~\ref{teo:deadker} it follows $\Psi(\gra) \neq \vuoto$, hence $\gra \in \grS$. Suppose conversely that $\gra \in \grS$, then by definition $\Phi^+(\gra) \cap \Psi \neq \vuoto$, therefore $\gra \in \weak$ by Theorem~\ref{teo:deadker} again. The last equalities follow immediately by the inclusions $\mZ \grS \subset \mZ \weak \subset \mZ \Psi \subset \mZ \grS$ and $\mN \grS \subset \mN\weak \subset \mZ \Psi \cap \mN \grD \subset \mN \grS$.
\end{proof}

\begin{remark}\label{rem:spherical roots}
Up to a twist, the set $\grS$ of Definition \ref{def:spherical-roots} is the set of the \textit{spherical roots} attached to $G/H$ following the theory of spherical varieties (see e.g. \cite[Theorem 1.3]{Kn3} for the general case, and \cite[Theorem 5.28]{Avd2} for the case where $H$ is a strongly solvable spherical subgroup). More precisely, if $\grS_{G/H}$ is the set of the spherical roots of $G/H$, then we have $\grS_{G/H} = -w_0(\grS)$: the twist appearing here is due to the equality of weight lattices $\calX(G/H) = -w_0\calX(B/H)$ proved in Proposition \ref{prop:B/H-toric} (for this reason the subgroup $H$ in \cite{Avd2} is assumed to be contained in the opposite Borel subgroup $B^{w_0}$).
\end{remark}

\begin{proposition}  \label{prop:semigruppi-attive}
Given $I \subset \scrD$, we have
$$\mN\weak_I = \mZ \Psi_I \cap \mN \grD = \mQ \Psi_I \cap \mN \grD.$$
\end{proposition}

\begin{proof}
By the definition of $\grS$ together with Proposition \ref{prop:spherical-roots} we have
$$
\mQ \Psi \cap \mN \grD = \mQ \Psi \cap \mN \grS = \mN \grS = \mN \weak.
$$
Therefore the claim is equivalent to the equalities
$$\mN\weak_I = \mZ \Psi_I \cap \mN \weak = \mQ \Psi_I \cap \mN \weak.$$

By the definition of $\weak_I$ we have the inclusions $\mN\weak_I \subset \mZ \Psi_I \cap \mN \weak \subset \mQ \Psi_I \cap \mN \weak$. Let $\gra \in \mQ \Psi_I \cap \mN\weak$ and write $\gra = \grb_1 + \ldots + \grb_n$ with $\grb_i \in \weak$. By Proposition~\ref{prop:radici-stabili}, we have $\langle \rho(D), \gra \rangle = 0$ for every $D \in \scrD \senza I$. Proceeding by induction on $n$, we show that we may choose $\grb_i$ in $\weak_I$ for all $i$. If $n=1$, then we have $\gra \in \weak$, hence $\gra \in \weak_I$ by Proposition~\ref{prop:radici-stabili}. Suppose $n >1$ and assume that $\grb_i \not \in \weak_I$ for some $i$. Then the same corollary implies $\langle \rho(D), \grb_i \rangle \neq 0$ for some $D \in \scrD \senza I$.

Suppose that $\langle \rho(D), \grb_i \rangle > 0$. As $\langle \rho(D), \gra \rangle = 0$, it follows $\langle \rho(D), \grb_j \rangle < 0$ for some $j \neq i$, hence 
$D = \delta(\grb_j)$ and $\grb_i + \grb_j \in \weak$ by Lemma~\ref{lemma:somma-vive}. Similarly, if $\langle \rho(D), \grb_i \rangle < 0$, then $D = \grd(\grb_i)$ and there is some $j \neq i$ with $\langle \rho(\grd(\grb_i)), \grb_j \rangle > 0$, hence $\grb_i + \grb_j \in \weak$ by Lemma~\ref{lemma:somma-vive} again. Therefore $\gra$ can be written as a sum of $n-1$ weakly active roots, and we conclude by the inductive assumption.
\end{proof}

%%%%%%%%%%%%%%%%%%%%%%%%%%%%%%%%%%%%%%%%%%%%%%%
\subsection{The root system associated to a $T$-orbit in $B/H$.}
%%%%%%%%%%%%%%%%%%%%%%%%%%%%%%%%%%%%%%%%%%%%%%%

Recall that a root subsystem $\Phi' \subset \Phi$ is \textit{closed} if, for all $\gra, \grb \in \Phi'$ such that $\gra + \grb \in \Phi$, it holds $\gra + \grb \in \Phi'$ as well. Given $I \subset \scrD$ we denote
$$
	\Phi_I = \mZ \Psi_I \cap \Phi, \qquad \qquad \Phi_I^{\pm} = \mZ \Psi_I \cap \Phi^\pm.
$$
It follows easily by its definition that $\Phi_I$ is a closed root subsystem of $\Phi$, and that $\Phi_I^+ \subset \Phi_I$ is a system of positive roots. We denote by $\grD_I \subset \Phi^+_I$ the corresponding basis and by $W_I$ the Weyl group of $\Phi_I$ (notice that in general $\grD_I \not \subset \grD$).

Proposition~\ref{prop:semigruppi-attive} readily implies the following property of the root system $\Phi_I$.

\begin{proposition}  \label{prop: I-rootsyst}
Given $I \subset \scrD$, we have the following equalities:
\begin{itemize}
	\item[i)] $\Phi^+_I = \mQ \Psi_I \cap \Phi^+ = \mN\weak_I \cap \Phi^+$;
	\item[ii)] $\weak_I = \Phi^+_I \cap \weak$ and $\Psi_I = \Phi^+_I \cap \Psi$.
\end{itemize}
In particular $\grD_I \subset \weak_I$, and $I$ is recovered by $\Phi_I$.
\end{proposition}

\begin{proof}
By Proposition~\ref{prop:semigruppi-attive} we get the equalities $\Phi^+_I = \mQ \Psi_I \cap \Phi^+ = \mN\weak_I \cap \Phi^+$, and the inclusion $\grD_I \subset \weak_I$ follows as well. Combining with Proposition~\ref{prop:radici-stabili} we get then $\weak_I = \Phi^+_I \cap \weak$, and since $\Psi_I = \weak_I \cap \Psi$ the last equality follows as well. The last claim follows by noticing that $I = \grd(\weak_I) = \grd(\Phi^+_I \cap \weak)$.
\end{proof}

We say that a root subsystem $\Phi' \subset \Phi$ is \textit{parabolic} if there exists $w \in W$ such that $w(\Phi')$ is generated by a subset of simple roots of $\Phi$. This is equivalent to the property that $\Phi' = \mQ \Phi' \cap \Phi$ (see \cite[Ch.~VI, \S~1, Proposition~24]{Bou}), therefore we get the following corollary.

%Proposition~\ref{prop: I-rootsyst} shows that $\Phi_I$ is a complete root subsystem of $\Phi$. Therefore we get the following corollary.

\begin{corollary} \label{cor: sistemi-parabolici}
Let $I \subset \scrD$, then $\Phi_I$ is a parabolic root subsystem of $\Phi$.
\end{corollary}

Since $\Phi_I \subset \Phi$ is a closed root subsystem, to every $I \subset \scrD$ we may also attach a reductive subgroup $G_I$ of $G$, namely the subgroup generated by $T$ together with the root subgroups $U_\gra$ with $\gra \in \Phi_I$. We denote by $B_I$ the Borel subgroup of $G_I$ associated with $\Phi^+_I$, that is $B_I = G_I \cap B$, and by $U_I$ its unipotent radical. The following proposition provides a first link between the root system $\Phi_I$ and the geometry of the corresponding $T$-orbit $\scrU_I$.

\begin{proposition} \label{prop:radici-reticoli}
Let $I \subset \scrD$, then $\ol{\scrU_I} = \ol{B_I H/H}$ and $\calX(\scrU_I) = \mZ \Phi_I + \calX(\scrU_\vuoto)$.
\end{proposition}

\begin{proof}
By Proposition~\ref{prop:Ualphaactionfinal} it follows that $\ol{\scrU_I}$ is $U_\gra$-stable for all $\gra \in \weak_I$. On the other hand by Proposition~\ref{prop: I-rootsyst} every $\gra \in \Phi^+_I \senza \weak_I$ acts trivially on $B/H$, therefore $\ol \scrU_I$ is $B_I$-stable. Moreover we have $\scrU_I \subset B_I H/H$ by Proposition \ref{prop:Torb_onD} and $B_I H/H\subseteq \overline{\scrU_I}$ by Propositions~\ref{prop:Ualphaactionfinal} and~\ref{prop: I-rootsyst}, and the first claim follows.

To show the second claim, it follows by the previous discussion that $\calX(\scrU_I) = \calX(B_I H/H)$. Reasoning as in \eqref{eqn:fiber-product}, the projection $B_I/ B_I \cap H \ra T/T_H$ induces a $T$-equivariant isomorphism
$$
	B_I/ B_I \cap H \simeq T \times^{T_H} U_I/U_I \cap H.
$$
Being a smooth and affine toric variety, $B_I/ B_I \cap H$ has trivial class group, therefore we get a surjective homomorphism $\phi : \calX(B_I H/H) \ra \mathrm{Div}^T(B_I/ B_I \cap H)$ defined by sending the character $\chi$ to the divisor of a rational $T$-eigenfunction of weight $\chi$ (see e.g. \cite[Theorem 4.1.3]{CLS}). Since the $T$-stable prime divisors of $B_I/ B_I \cap H$ correspond to the $T_H$-stable hyperplanes of the fiber $U_I/U_I \cap H$ (which is a affine space), it follows that the kernel of $\phi$ equals $\calX(T/T_H) = \calX(\scrU_\vuoto)$. On the other hand by Corollary \ref{cor:weak-vs-active} every $T$-stable prime divisor on $B_I H/H$ has a global equation of weight $-\gra$ with $\gra \in \Psi_I$, therefore $\calX(\scrU_I) = \calX(B_I H/H)$ is generated by $\calX(\scrU_\vuoto)$ together with $\mZ \Psi_I = \mZ \Phi_I$.
\end{proof}

\begin{corollary}
Let $I \subset \scrD$, then the following inequalities hold
\[
	|I| \leq \rk \Phi_I \leq \dim(T/T_H) + |I|
\]
\end{corollary}

\begin{proof}
This follows from the equality $\rk \calX(\scrU_I) = \dim(\scrU_I) = \dim(T/T_H) + |I|$ together with the inclusion $\Psi_I \subset \Phi_I$.
\end{proof}

\begin{remark}	\label{oss:ranks}
By the definition of $\Phi_I$ we have $\rk (\Phi_I) = \dim(\mQ \Psi_I)$. In some cases this rank is easily computable:
\begin{itemize}
	\item[-] $0 \leq \rk(\Phi_I) \leq |\grS|$.
	\item[-] $\rk(\Phi_I) = 0$ if and only if $I = \vuoto$.
	\item[-] $\rk(\Phi_I) = |\grS|$ if and only if $I = \scrD$.
	\item[-] If $I = \{D\}$ then $\rk(\Phi_I) = |\grd^{-1}(D)|$ (this follows by \cite[Corollary 1]{Avd1}).
	\item[-] If $T \subset H$ then $\rk(\Phi_I) = |I|$ for all $I \subset \scrD$ (this follows by \cite[Theorem 1]{Avd1}).
\end{itemize}
\end{remark}

\subsection{Explicit description of the basis $\grD_I$}

%In the present section we compute some examples. In particular, we will give an explicit description of the basis $\grD_I \subset \Phi_I$ in the case of a simply laced group and of a 

By making use of the classification of the active roots, we will give in this subsection some more explicit description of the root systems $\Phi_I$ and of their bases $\grD_I$. We will not use these results until the last section of the paper, where we will prove a bound for the number of $B$-orbits in $G/H$.

Suppose that $H \subset B$ is a spherical subgroup of $G$ (possibly not containing $T$). In the following theorem we recall Avdeev's classification of the active roots, as well as some of their properties (see \cite{Avd1}, Proposition 3, Corollary 6, Theorem 3, and Lemma 10), and deduce some corollaries.

\begin{theorem} [{\cite{Avd1}}]	\label{teo:avdeev}
Let $\grb \in \Psi$, then the following properties hold.
\begin{itemize}
	\item[i)] There exists a unique simple root $\pi(\grb) \in \supp(\grb)$ with the following property: if $\grb = \grb_1 + \grb_2$ for some $\grb_1, \grb_2 \in \Phi^+$, then $\grb_1 \in \Psi$ if and only if $\pi(\grb) \not \in \supp(\grb_1)$.
	\item[ii)] The map $\grb' \mapsto \pi(\grb')$ induces a bijection between $F(\grb)$ and $\supp(\grb)$.
	\item[iii)] If $\grb_1, \grb_2 \in \Psi$ and $\pi(\grb_1) = \pi(\grb_2)$, then $\grd(\grb_1) = \grd(\grb_2)$.
	\item[iv)] The active root $\grb$ appears in Table~\ref{tab: active roots}\footnote{Notice that when $\supp(\grb)$ is of type $\sfF_4$ our enumeration of the set $\supp(\grb)$ in Table~\ref{tab: active roots} differs from the one in \cite{Avd1}: following \cite{Bou} for us $\alpha_1$ and $\alpha_2$ are the long simple roots.}, and $[\grb : \pi(\grb)] = 1$.
\end{itemize}
\end{theorem}

\begin{table}
\begin{center}
\begin{tabular}{|c||c|c|}
\hline
Type & $\supp(\grb)$ & $\grb$ \\
\hline
\hline
1. & any of rank $n$ & $\alpha_1+\cdots+\alpha_n$ \\
\hline
2. & $\mathsf{B}_n$ & $\alpha_1+\cdots+\alpha_{n-1} + 2\alpha_n$  \\
\hline
3. & $\mathsf{C}_n$ & $2\alpha_1+\cdots+2\alpha_{n-1} + \alpha_n$  \\
\hline
4. & $\mathsf{F}_4$ & $\alpha_1 + \alpha_2 + 2\gra_3 + 2\gra_4$ \\
\hline
5. & $\mathsf{G}_2$ & $2\alpha_1+\alpha_2$ \\
\hline
6. & $\mathsf{G}_2$ & $3\alpha_1+\alpha_2$ \\
\hline
\end{tabular}
\end{center}
\caption{Active roots} \label{tab: active roots}
\end{table}

We will call the integer appearing in the first column of Table~\ref{tab: active roots} the \textit{type} of an active root.

Let $\grb \in \Psi$, we say that a subset $A \subset \supp(\grb)$ is \textit{connected} (resp. \textit{co-connected}) if $A$ (resp. $\supp(\grb) \senza A$) is connected as a set of vertices in the Dynkin diagram of $\Phi$. If $\grb' \in F(\grb)$, it follows by a direct inspection of Table \ref{tab: active roots} that $\supp(\grb') \subset \supp(\grb)$ is co-connected, whereas Theorem~\ref{teo:avdeev} i) shows that, if $\grb' \neq \grb$, then $\pi(\grb) \not \in \supp(\grb')$.

Notice that Theorem~\ref{teo:avdeev} allows to construct the whole family $F(\grb)$ from $\grb$ and $\pi(\grb)$. In particular we have the following property.

\begin{corollary} \label{cor:support-family}
Let $\grb \in \Psi$ and let $A \subset \supp(\grb)$. Then $A = \supp(\grb')$ for some $\grb' \in F(\grb) \senza \{\grb\}$ if and only if $A$ is connected and co-connected, and $\pi(\grb) \not \in A$.
\end{corollary}

\begin{proof}
As already noticed, every active root $\grb' \in F(\grb) \senza \{\grb\}$ satisfies the properties of the claim. The corollary follows then by Theorem~\ref{teo:avdeev} ii) by noticing that the number of subsets $A\subset \supp(\grb)$ which are connected and co-connected such that $\pi(\grb) \not \in A$ is the number of edges of the Dynkin diagram of $\supp(\grb)$. Indeed, associating to $A$ the unique edge connecting a simple root of $A$ to a simple root of $\supp(\grb)\senza A$ is a bijection. On the other hand, the number of such edges is precisely $|\supp(\grb)|-1$.
\end{proof}

In particular we get the following characterization of the pairs $(\gra, \grb) \in \weak \times \Psi$ with $\grb \in \Psi(\gra)$.

\begin{corollary} \label{cor:weak-family}
Let $\gra \in \Phi^+$ and $\grb \in \Psi$, and suppose that $\gra < \grb$. Then $\grb \in \Psi(\gra)$ if and only if $\pi(\grb) \in \supp(\gra)$.
\end{corollary}

\begin{proof}
Since $\grb \in \Phi^+(\gra)$, by Theorem~\ref{teo:deadker} it follows that $\gra \in \weak$. If $\supp(\gra) \neq \supp(\grb)$, let $A \subset \supp(\grb) \senza \supp(\gra)$ be a connected component. Notice that $A$ is both connected and co-connected in $\supp(\grb)$, therefore by Corollary~\ref{cor:support-family} there is $\grb' \in \Psi$ with $\supp(\grb') = A$. On the other hand $\langle \grb', \gra^\vee \rangle < 0$, therefore $\gra + \grb' \in \Phi^+$, and since $\gra + \grb' \leq \grb$ by Theorem \ref{teo:deadker} we still have $\gra + \grb' \in \weak$. On the other hand, $\grb -\grb' \in \Phi^+$ by Lemma \ref{lemma:differenza-attive-avdeev}, hence $\pi(\grb) \not \in \supp(\grb')$ by Theorem \ref{teo:avdeev}. Therefore $\pi(\grb) \in \supp(\gra)$ if and only if $\pi(\grb) \in \supp(\gra+\grb')$. 

It follows that the claim holds for $\gra$ if and only if it holds for $\gra+\grb'$. Therefore we may assume that $\supp(\gra) =\supp(\grb)$, in which case the claim can be easily checked by making use of Theorem~\ref{teo:avdeev} and Table~\ref{tab: active roots}:
\begin{itemize}
	\item[i)] If $\grb$ is of type 1, then $\gra = \grb$, which is absurd.
	\item[ii)] If $\grb$ is of type 2, then $\gra = \grb - \gra_n$, and $\gra_n \in \Psi$.
	\item[iii)] If $\grb$ is of type 3, then $\gra = \grb - (\gra_1 + \ldots + \gra_i)$ for some $i < n$, and $\gra_1 + \ldots + \gra_i \in \Psi$ for all $i < n$.
	\item[iv)] If $\grb$ is of type 4, then either $\gra = \grb - (\gra_3 + \gra_4)$ or $\gra = \grb - \gra_4$, and 	both $\gra_3 + \gra_4$ and $\gra_4$ are active roots.
	\item[v)] If $\grb$ is of type 5, then $\gra = \grb - \gra_1$, and $\gra_1 \in \Psi$.
	\item[vi)] If $\grb$ is of type 6, then either $\gra = \grb - \gra_1$ or $\gra = \grb - 2\gra_1$, and $\gra_1 \in \Psi$.
	\qedhere.
\end{itemize}
\end{proof}

In particular, the previous corollary shows that $\grb \in \Psi(\pi(\grb))$ for all $\grb \in \Psi$. In this case we can be even more precise.

\begin{corollary}	\label{cor:massimali}
Let $\grb \in \Psi$ and $\grb' \in F(\grb)$, then $\grb'$ is maximal in $F(\grb) \senza \{\grb\}$ if and only if $\supp(\grb')$ is a connected component of $\supp(\grb) \senza \{\pi(\grb)\}$ if and only if $\pi(\grb)$ and $\pi(\grb')$ are non-orthogonal. Moreover
$$\pi(\grb) = \grb - \sum_{\grg \in \max (F(\grb) \senza \{\grb\}) } a_{\grb,\grg} \grg$$
where $\max (F(\grb) \senza \{\grb\})$ denotes the set of the maximal elements in $F(\grb) \senza \{\grb\}$ and where $a_{\grb,\grg} \in \mN$ is the maximum such that $a_{\grb,\grg} \grg < \grb$.
\end{corollary}

\begin{proof}
Notice that every connected component of $\supp(\grb) \senza \{\pi(\grb)\}$ is co-connected. The first claim follows then by Corollary~\ref{cor:support-family}, together with Proposition~\ref{prop:family} ii).

Let $\gra' \in \supp(\grb')$ non-orthogonal to $\pi(\grb)$, notice that such root is unique since $\supp(\grb)$ contains no loops. Then $\supp(\grb') \senza \{\gra'\}$ is connected and co-connected in $\supp(\grb)$, and since it does not contain $\pi(\grb)$ by Corollary~\ref{cor:support-family} there is $\grb'' \in F(\grb)$ such that $\supp(\grb'') = \supp(\grb') \senza \{\gra'\}$. Since $\supp(\grb'') \subset \supp(\grb')$, Proposition~\ref{prop:family} implies that $\grb'' \in F(\grb')$ and $\grb'-\grb'' \in \Phi^+$. On the other hand by Theorem~\ref{teo:avdeev} i) $\pi(\grb') \in \supp(\grb') \senza \supp(\grb'')$, therefore it must be $\pi(\grb') = \gra'$.

The last claim can be easily deduced by the previous discussion together with a direct inspection based on Theorem~\ref{teo:avdeev} and Table~\ref{tab: active roots}.
\end{proof}

Another technical property that we will need and that we can deduce from Theorem~\ref{teo:avdeev} is the following.

\begin{lemma}	\label{lemma:radici-adiacenti}
Let $\grb \in \Psi$ and let $\grb_1, \grb_2 \in F(\grb)$. If $\supp(\grb_1) \cup \supp(\grb_2)$ is connected, then $\grb_1$ and $\grb_2$ are comparable.
\end{lemma}

\begin{proof}
By Corollary~\ref{cor:support-family} the complementary
$$
\supp(\grb) \senza (\supp(\grb_1)\cup\supp(\grb_2)) = (\supp(\grb)\senza\supp(\grb_1))\cap (\supp(\grb)\senza\supp(\grb_2))
$$
is the intersection of two connected subsets of $\supp(\grb)$, so it is connected since the Dynkin diagram of $\supp(\grb)$ has no loops. Therefore $\supp(\grb_1)\cup\supp(\grb_2)$ is connected and co-connected, and by Corollary~\ref{cor:support-family} there exists $\gamma\in F(\grb)$ such that $\supp(\gamma)=\supp(\grb_1)\cup\supp(\grb_2)$.

Suppose that $\gamma$ is different both from $\grb_1$ and $\grb_2$. For $i = 1,2$, we have then $\grb_i\in F(\gamma)$ by Proposition~\ref{prop:family} ii), and $\gamma-\grb_i\in\Phi^+$ by Lemma~\ref{lemma:differenza-attive-avdeev}. Therefore by Theorem~\ref{teo:avdeev} i) $\pi(\gamma)\notin\supp(\grb_i)$, which is absurd by the definition of $\grg$. Therefore $\gamma$ is either $\grb_1$ or $\grb_2$, hence $\grb_1$ and $\grb_2$ are comparable by Proposition~\ref{prop:family} ii).
\end{proof}

Let $I\subset \scrD$ and $\grb \in \Psi_I$. Denote $F_I(\grb) = (F(\grb) \cap \Psi_I) \senza \{\grb\}$ and define
$$\grb^\sharp_I = \grb - \sum_{\grg \in \max (F_I(\grb)) } a_{\grb,\grg} \grg$$
where $\max (F_I(\grb))$ is the set of the maximal elements in $F_I(\grb)$ and where $a_{\grb,\grg} \in \mN$ is the maximum such that $a_{\grb,\grg} \grg < \grb$. By Theorem~\ref{teo:deadker} it holds $\grb_I^\sharp \in \weak_I$, and $\grd(\grb_I^\sharp) = \grd(\grb)$: indeed, being pairwise incomparable, by Lemma~\ref{lemma:radici-adiacenti} the elements in $\max (F_I(\grb))$ have pairwise disjoint support.

\begin{theorem}	\label{teo:DeltaI-rango-massimo}
Let $I \subset \scrD$, then $\grD_I = \{\grb^\sharp_I \st \grb \in \Psi_I\}$.
\end{theorem}

\begin{proof}
Suppose that $\gra \in \grD_I$, then by Corollary~\ref{cor:weak-as-differences} we may write $\gra = \grb - (\grb_1 + \ldots + \grb_p)$ for some $\grb \in \Psi_I$ and some $\grb_1, \ldots, \grb_p \in F_I(\grb)$. Suppose that $p > 0$, otherwise there is nothing to show. We claim that
\begin{equation}	\label{eqn:formula-beta_I}
\grb_1 + \ldots + \grb_p \leq \sum_{\grg \in \max (F_I(\grb)) } a_{\grb,\grg}\grg.
\end{equation}

The inequality is clear if the supports of $\grb_1, \ldots, \grb_p$ are pairwise disconnected, e.g. if the active root $\grb$ is of type 1. Suppose that this is not the case and that (up to reordering the indices) $\supp(\grb_1) \cap \supp(\grb_2) \neq \vuoto$. Then $\grb_1$ and $\grb_2$ are comparable by Lemma~\ref{lemma:radici-adiacenti}, and we may assume that $\grb_1 \leq \grb_2$. Since $\grb_1 + \grb_2 < \grb$, arguing with Theorem~\ref{teo:avdeev} and Table~\ref{tab: active roots} we have the following possibilities:
\begin{itemize}
	\item[i)] $\grb$ is of type 2, and $\grb_1 = \grb_2 = \gra_n$;
	\item[ii)] $\grb$ is of type 3, and there are indices $i,j$ with $1 \leq i \leq j < n$ such that $\grb_1 = \gra_1 + \ldots + \gra_i$ and $\grb_2 = \gra_1+ \ldots + \gra_j$;
	\item[iii)] $\grb$ is of type 4, and there are indices $i,j$ with $3 \leq i \leq j \leq 4$ such that $\grb_1 = \gra_j + \ldots + \gra_4$ and $\grb_2 = \gra_i+ \ldots + \gra_4$;
	\item[iv)] $\grb$ is of type 5 or 6, and $\grb_1 = \grb_2 = \gra_1$.
\end{itemize}
If $\grg \in \max(F_I(\grb))$ is such that $\grb_2 \leq \grg$, in all these cases we get that $\grb_1 + \grb_2 \leq a_{\grb,\grg} \grg$.

By making use of Theorem \ref{teo:avdeev} it is easy to see that for all $\grb' \in F(\grb) \senza \{\grb\}$ the family $F(\grb')$ is totally ordered by the dominance order. On the other hand, we see by Table \ref{tab: active roots} that the coefficient of the active root $\grb$ along a simple root can be at most 2, unless $\grb$ is of Type 6. Provided that $\grb$ is not of type 6, it follows that no root $\grb_i$ with $i > 2$ is in $F(\grg)$, since otherwise $\grb_1 + \grb_2 + \grb_i \not \leq \grb$, and inequality \eqref{eqn:formula-beta_I} follows by summing up all the roots with intersecting supports. Finally, the inequality is immediately checked if $\grb = 3\gra_1 + \gra_2$ is of type 6, in which case $\gra = \gra_2$ and $\grb_1 = \grb_2 = \grb_3 = \gra_1$.

As a consequence of \eqref{eqn:formula-beta_I}, we get the inequality $\grb^\sharp_I \leq \gra$. Since by assumption $\gra$ is minimal in $\Phi^+_I$ and since $\grb^\sharp_I \in \Phi^+_I$, it follows that $\gra = \grb^\sharp_I$. Therefore we have proved the inclusion $\grD_I \subset \{\grb^\sharp_I \st \grb \in I\}$. We now show that for all $\grb \in \Psi$ the element $\grb_I^\sharp$ is indecomposable in $\Phi_I^+$.

Suppose that $\grb \in \Psi_I$ and let $\grg_1, \ldots, \grg_n \in \Psi_I$ be such that
$$
	\grb_I^\sharp = (\grg_1)_I^\sharp + \ldots + (\grg_n)_I^\sharp.
$$
By Definition~\ref{def:radici-toriche}, for all $i$ the evaluation of $(\grg_i)_I^\sharp$ takes exactly one negative value on $\scrD$. As $\grd(\grb_I^\sharp) = \grd(\grb)$, there exists $\grg \in \{\grg_1, \ldots, \grg_n\}$ with $\grd(\grb) = \grd(\grg)$. Therefore $\grg_I^\sharp \leq \grb_I^\sharp < \grb$, and $\pi(\grg) = \pi(\grg_I^\sharp) \in \supp(\grb)$. By \cite[Corollary 11]{Avd1} we get then $\pi(\grb) = \pi(\grg_I^\sharp)$, and Corollary~\ref{cor:weak-family} shows that $\grb \in \Psi(\grg_I^\sharp)$. Therefore there are $\grb_1, \ldots, \grb_m \in F(\grb)$ such that $\grb = \grg_I^\sharp + \grb_1+ \ldots + \grb_m$. Then $\langle \rho(\grd(\grb_i)),\grg_I^\sharp \rangle > 0$ for all $i=1, \ldots, m$, and by Proposition~\ref{prop:radici-stabili} it follows $\grd(\grb_i) \in I$, that is $\grb_i \in \Psi_I$. Therefore the definition of $\grb_I^\sharp$ implies $\grb_I^\sharp \leq \grg_I^\sharp$, that is $n=1$ and $\grb_I^\sharp = \grg_I^\sharp$.
\end{proof}

We assume for the rest of the section that $H \subset G$ is a spherical subgroup such that $T \subset H \subset B$. Then $\grd_{|\Psi} : \Psi \ra \scrD$ is bijective by Corollary~\ref{cor:delta-surjective}, and $\pi : \Psi \ra \grD$ is injective by Theorem~\ref{teo:avdeev} iii). We will identify the set of divisors $\scrD$ with the set of active roots $\Psi$.

Define a graph $\mathcal G(\Psi)$ with set of vertices $\Psi$ as follows: two active roots $\grb, \grb'$ are connected by an edge if and only if they are not strongly orthogonal (namely, neither their sum nor their difference is in $\Phi$). Given $I \subset \Psi$, we denote by $\mathcal G(I)$ the associated subgraph of $\mathcal G(\Psi)$. Notice that, if $\grb, \grb'$ belong to the same connected component of $\calG(I)$, then they belong to the same irreducible component of $\Phi_I$: if indeed $\{\grb, \grb'\} \subset I$ is an edge of $\mathcal G(I)$, then either $\grb + \grb' \in \Phi_I$ or $\grb - \grb' \in \Phi_I$.

If $I \subset \Psi$ we denote by $\Phi'_I$ the root system generated by the simple roots $\pi(\grb)$ with $\grb \in I$, and by $W'_I$ the Weyl group of $\Phi'_I$.

\begin{proposition}	\label{prop:intervalli}
Suppose that $T \subset H$, let $I \subset \Psi$ and suppose that $\pi(I)$ is connected. Then the followings hold:
\begin{itemize}
	\item[i)] $\Phi_I$ is an irreducible parabolic subsystems of $\Phi$ of rank $|I|$.
	\item[ii)] If $\Phi$ is simply laced, then $\Phi_I$ and $\Phi'_I$ are isomorphic root systems.
	\item[iii)] If $\Phi$ is not simply laced and $\pi(I)$ contains both long and short roots, then $\Phi_I$ and $\Phi'_I$ are isomorphic root systems.
\end{itemize}
\end{proposition}

\begin{proof}
i) Let $\grb, \grb' \in I$, denote $\gra =\pi(\grb)$ and $\gra' =\pi(\grb')$ and suppose $\langle \gra', \gra^\vee\rangle < 0$. Suppose that there is some active root whose family contains both $\grb$ and $\grb'$: then $\grb$ and $\grb'$ are comparable by Lemma ~\ref{lemma:radici-adiacenti}, hence $\{\grb,\grb'\}$ is an edge of $\calG(I)$. Therefore $\grb, \grb'$ belong to the same component of $\Phi_I$.

Suppose that there is no active root whose family contains both $\grb$ and $\grb'$. We claim that in this case $\supp(\grb) \cap \supp(\grb') = \vuoto$. Indeed, if this intersection is not empty then $\gra\in\supp(\grb')$ or $\gra'\in\supp(\grb)$ (because the Dynkin diagram of $G$ has no loops). Assume that $\gra'\in\supp(\grb)$, then by Theorem~\ref{teo:avdeev} ii) there is $\grb''\in F(\grb)$ with $\pi(\grb'')=\gra'$, which implies $\grb''=\grb'$ by the injectivity of $\pi$. Therefore $\grb'\in F(\grb)$, contradicting the assumption.

Therefore $\supp(\grb) \cap \supp(\grb') = \vuoto$, and since $\gra \in \supp(\grb)$ and $\gra' \in \supp(\grb')$ are non-orthogonal Theorem~\ref{teo:avdeev} iv) implies that $\langle \grb', \grb^\vee \rangle = \langle \gra', \gra^\vee\rangle < 0$. Therefore $\grb + \grb' \in \Phi^+_I$, hence $\grb, \grb'$ belong to the same component of $\Phi_I$. It follows that $\Phi_I$ is an irreducible subsystem of $\Phi$ (whose rank is $|I|$ by Remark~\ref{oss:ranks}), and it is parabolic by Corollary~\ref{cor: sistemi-parabolici}. 

ii) By Theorem~\ref{teo:DeltaI-rango-massimo} $\grD_I = \{\grb^\sharp_I \st \grb \in I\}$ is a basis of $\Phi_I$. Let $\grb \in I$. %and suppose that $\pi(\grb)$ is not extremal in $\pi(I)$. Then by Corollary~\ref{cor:massimali} $I$ contains all the maximal roots $\grb' \in F(\grb) \senza \{\grb\}$, and $\grb_I^\sharp = \pi(\grb)$. Similarly, if $\pi(\grb)$ is extremal in $\pi(I)$,
By Corollary~\ref{cor:massimali}, if $\gra'\in\pi(I)$ adjacent to $\pi(\grb)$ then the root $\beta'\in I$ such that $\pi(\beta')=\gra$ is maximal in $F(\grb)\senza\{\grb\}$, hence $\beta'\in\max(F_I(\grb))$, which implies that $\grb^\sharp_I$ doesn't have $\gra'$ in its support. In other words $\supp(\grb_I^\sharp) \cap \pi(I)$ contains $\pi(\grb)$ but none of the roots of $\pi(I)$ adjacent to it: thanks to the fact that $\supp(\grb_I^\sharp) \cap \pi(I)$ is connected, we deduce that $\supp(\grb_I^\sharp) \cap \pi(I) = \{\pi(\grb)\}$. Notice also that all elements of $\supp(\grb_I^\sharp)\senza\{\pi(\grb)\}$ are orthogonal to the elements of $\pi(I)\senza\{\pi(\grb)\}$, otherwise the Dynkin diagram of $\Phi$ would have a loop. Since all the roots in $\Phi$ have the same length, it follows that the root systems generated by $\grD_I$ and by $\pi(I)$ are isomorphic.

iii) Let $\grD_0 \subset \grD$ be the connected component containing $\pi(I)$. Denote $\Phi_0$ the corresponding irreducible subsystem of $\Phi$ and enumerate $\grD_0 = \{\gra_1, \ldots, \gra_n\}$ as in \cite{Bou}. We also enumerate $I$ by setting $\grb_i = \pi^{-1}(\gra_i)$ for all $\gra_i \in \pi(I)$. If $\pi(I) = \grD_0$, then the claim is clear, since i) implies then $\Phi_I = \Phi'_I = \Phi_0$. In particular we may assume that $\Phi'_I$ is not of type $\sfF_4$ nor of type $\sfG_2$.

Suppose first that $\Phi_0$ is of type $\sfB_n$ or $\sfC_n$. Then $\pi(I) = \{\gra_p, \ldots, \gra_n\}$ for some $p$ with $1 < p < n$. Notice that the claim follows once we prove that $\Phi_I$ contains roots of different length: indeed by i) $\Phi_I$ and $\Phi'_I$ are both irreducible parabolic subsystems of $\Phi_0$ of rank $|I|$. By Corollary~\ref{cor:massimali} it follows that $I$ contains all the maximal elements in $F(\grb_n) \senza \{\grb_n\}$, therefore $(\grb_n)_I^\sharp = \pi(\grb_n) = \gra_n$ and we get $\gra_n \in \Phi^+_I$.

Suppose that $\Phi_0$ is of type $\sfB_n$, then $\gra_n$ is short. On the other hand a positive root $\grg$ with $[\grg : \gra_{n-1}] = 1$ is short if and only if $\grg = \gra_i + \ldots + \gra_n$ for some $i \leq n$, therefore either $\grb_{n-1}$ is long or $\grb_{n-1}-\gra_n$ is long, and the claim follows since $\Phi_I$ is a parabolic subsystem of $\Phi_0$.

Suppose that $\Phi_0$ is of type $\sfC_n$, then $\gra_n$ is long. On the other hand a positive root $\grg$ with $[\grg:\gra_{n-1}] = 1$ is necessarily short, and the claim follows.

Suppose now that $\Phi_0$ is of type $\sfF_4$, then by the assumption $\pi(I) \supset \{\gra_2, \gra_3\}$. If $\bigcup_{\grb \in I} \supp(\grb)$ is of type $\sfB$ or $\sfC$ the same arguments as before apply, suppose that this is not the case. A direct inspection based on Theorem~\ref{teo:avdeev} shows that $\grb_3$ is necessarily a short root, and that either $\grb_2$ is a long root or $\grb_2 - \grb_3$ is a long root. Therefore $\Phi_I$ contains roots of different lengths, hence it is of type $\sfB_{|I|}$ or $\sfC_{|I|}$. If $|I| = 2$ then the claim follows.

Suppose that $|I| = 3$. If $\grb_1 \in I$, then it follows by Corollary \ref{cor:massimali} that $I$ contains all the maximal elements in $F(\grb_1) \senza \{\grb_1\}$ and in $F(\grb_2) \senza \{\grb_2\}$. Therefore $(\grb_i)^\sharp_I = \pi(\grb_i) = \gra_i$ for $i=1,2$, and it follows that $\grD_I$ contains two long roots, hence $\Phi_I$ is of type $\sfB_3$. Similarly, if $\grb_4 \in I$ we conclude that $\grD_I$ contains two short roots, hence $\Phi_I$ is of type $\sfC_3$.
\end{proof}

%%%%%%%%%%%%%%%%%%%%%%%%%%%%%%%%%%%
%%%%%%%%%%%%%%%%%%%%%%%%%%%%%%%%%%%
\section{Combinatorial parameters for the orbits of $B$ on $G/H$}\label{s:Borbits}
%%%%%%%%%%%%%%%%%%%%%%%%%%%%%%%%%%%
%%%%%%%%%%%%%%%%%%%%%%%%%%%%%%%%%%%

We are now ready to give a combinatorial parametrization of the $B$-orbits in $G/H$, where $H$ is a spherical subgroup of $G$ contained in $B$, in terms of the root systems introduced in the previous sections.

%%%%%%%%%%%%%%%%%%%%%%%%%%%%%%%%%%%%%%%%%%%%%%%
\subsection{Reduced and extended pairs}
%%%%%%%%%%%%%%%%%%%%%%%%%%%%%%%%%%%%%%%%%%%%%%%

Consider the projection $G/H \ra G/B$. By the Bruhat decomposition, every $B$-orbit $\scrO$ in $G/H$ uniquely determines a Weyl group element, namely the element $w\in W$ such that the image of $\scrO$ in $G/B$ is the Schubert cell $BwB/B$. Given $w \in W$ and $I \subset \scrD$, notice that $w\scrU_I$ is a well defined $T$-orbit in $G/H$, mapping on the $T$-fixed point $wB/B \in G/B$. Therefore to every pair $(w,I)$ with $w \in W$ and $I \subset \scrD$, we may associate a $B$-orbit in $G/H$ by setting $\scrO_{w,I} = Bw \scrU_I$.

Notice that every $B$-orbit in $G/H$ is of the shape $\scrO_{w,I}$ for some $w \in W$ and some $I \subset \scrD$. Suppose indeed that $\scrO$ is a $B$-orbit in $G/H$ which projects on the Schubert cell $BwB/B$. Then, reasoning as in Proposition~\ref{prop:corrispondenza-divisori}, the intersection $w^{-1} \scrO \cap B/H$ is a $(B \cap B^w)$-orbit in $B/H$, hence it is $T$-stable. If $\scrU_I$ is any $T$-orbit in $w^{-1} \scrO \cap B/H$, we have then the equality $\scrO = \scrO_{w,I}$. While $w$ is uniquely determined by $\scrO$, in general there are several choices for $I \subset \scrD$.

We will say that $I \subset \scrD$ is a \textit{representative} for a $B$-orbit $\scrO$ if $\scrO = \scrO_{w,I}$ for some $w \in W$. The goal of the following theorem is to show that there are canonical minimal and maximal representatives for the $B$-orbits in $G/H$, and to give combinatorial characterizations of such pairs.

\begin{theorem}	\label{teo:parametrizzazione}
Let $\scrO$ be a $B$-orbit in $G/H$. There exist a unique minimal representative $\m$ and a unique maximal representative $\M$ for $\scrO$, and $I$ is a representative for $\scrO$ if and only if $\m \subset I \subset \M$. Suppose moreover that $\scrO = \scrO_{w,I}$ and denote
$$I(w) = \delta\big(\Theta_I \senza \Phi^+(w) \big).$$
\begin{itemize}
\item[i)]  We have the equalities $\m = I \senza I(w)$ and $\M = I \cup I(w)$.
\item[ii)] The following formulae hold:
\begin{gather*}
	\dim(\scrO) = \dim(B/H) + l(w) - |\scrD \senza \M|,\\
	\rk(\scrO) = \rk(B/H) + |\m|.
\end{gather*}
More precisely, $\calX_B(\scrO) = w \calX_T(\scrU_\m)$, and in particular $w(\Phi_\m) \subset \calX_B(\scrO)$.
\end{itemize}
\end{theorem}

\begin{proof}
Let $w \in W$ and $I \subset \scrD$ be such that $\scrO = \scrO_{w,I}$ and denote $Z = w^{-1}\scrO \cap B/H$. By the discussion at the beginning of this section, the first claim is equivalent to the fact that $Z$ contains a unique minimal $T$-orbit and a unique maximal $T$-orbit (ordered via inclusion of closures). Since $Z = (B \cap B^w) \scrU_I$, it follows that $Z$ is homogeneous under the action of a connected solvable group, and reasoning as in Corollary~\ref{cor:torica} it follows that $Z$ is an irreducible $T$-stable affine subvariety of $B/H$.

In particular, $Z$ is itself a toric variety under the action of $T$. Since it is irreducible, there exists a unique maximal $T$-orbit $\scrU_\M \subset Z$, and since it is affine there exists a unique closed $T$-orbit $\scrU_\m \subset Z$. To show that $\scrU_J \subset Z$ for every $J$ with $\m \subset J \subset \M$, notice that $Z$ is open in its closure. Therefore the boundary $\ol Z \senza Z$ is closed and $T$-stable, and if $\scrU_J \subset \ol Z \senza Z$ then it must be $\m \not \subset J$ because $\scrU_\m \subset Z$. Therefore $Z$ is the union of all $\scrU_J \subset B/H$ such that $\scrU_\m \subset \ol \scrU_J \subset \ol \scrU_\M$.

i) First at all, notice that $I(w) \subset \M$. Indeed $B\cap B^w$ is the product of the root subgroups $U_\gra$ with $\gra \in \Phi^+ \senza \Phi^+(w)$. If $\gra \in \Theta_I \senza \Phi^+(w)$, by Proposition \ref{prop:Ualphaactionfinal} it follows then $\scrU_{I\cup \grd(\gra)} \subset  U_\gra \scrU_I  \subset (B \cap B^w) \scrU_I$, hence $I \cup\{\grd(\gra)\} \subset \M$ by the maximality of $\M$.

Similarly, notice that $\m \subset I \senza I(w)$: if indeed $\gra \in \Theta_I \senza \Phi^+(w)$ and $\grd(\gra) \in I$, then $\scrU_{I\senza \grd(\gra)} \subset  U_\gra \scrU_I  \subset (B \cap B^w) \scrU_I$, hence  $\m \subset I \senza \{\grd(\gra)\}$ by the minimality of $\m$. At this point, to conclude it is enough to show the equality $\M = \m \cup I(w)$. We have already proved the inclusion $\m \cup I(w)\subset \M$. Since $\m(w) \subset I(w)$, the reverse inclusion $\M \subset \m \cup I(w)$ follows if we prove that $\M \senza \m \subset \m(w)$.

We claim that
\begin{equation}	\label{eqn:formulone}
	\forall D \in \M \senza \m \qquad \exists \grb \in \Psi_D \quad \exists \grg \in \mN \Psi_\m \quad \st \quad \grb - \grg \in \mN (\weak \senza \Phi^+(w)).
\end{equation}
Indeed, given $\gra \in \Phi^+$, we have $U_\gra \subset B \cap B^w$ if and only if $\gra \not \in \Phi^+(w)$. Since $\scrU_\M \subset (B \cap B^w) \scrU_\m$, it follows then by Proposition~\ref{prop:Ualphaactionfinal} that there is a sequence $\gra_1, \ldots, \gra_n \in \weak \senza \Phi^+(w)$ such that 
$\M \senza \m = \{\grd(\gra_1), \ldots, \grd(\gra_n)\}$ and $U_{\gra_n} \cdots U_{\gra_1} \scrU_\m \supset \scrU_M$. Moreover, every root $\gra_i$ has to be activated by $\m \cup \{\grd(\gra_1), \ldots, \grd(\gra_{i-1})\}$, that is: for all $i = 1, \ldots, n$ there are $\grg_i \in \mN \Psi_\m$, $\grb_{i,1} \in \Psi_{\grd(\gra_1)}$, $\ldots$, $\grb_{i,i-1} \in \Psi_{\grd(\gra_{i-1})}$ and non-negative integers $a_{i,1}$, $\ldots$, $a_{i,i-1}$ such that
\begin{align*}
	\grb_1 := \gra_1 + \grg_1 \in \Psi_{\grd(\gra_1)}	\\
	\grb_2 := \gra_2 + \grg_2 + a_{1,1} \grb_{1,1} \in \Psi_{\grd(\gra_2)} \\
	\grb_3 := \gra_3 + \grg_3 + a_{2,1} \grb_{2,1} + a_{2,2} \grb_{2,2} \in \Psi_{\grd(\gra_3)} \\
	\qquad \qquad \cdots \cdots \cdots \cdots \cdots \cdots \cdots \cdots \qquad \qquad \\
	\grb_n := \gra_n + \grg_n + a_{n,1} \grb_{n,1} + \ldots + a_{n,n-1} \grb_{n,n-1} \in \Psi_{\grd(\gra_n)}
%		\grb_n := \gra_n + \grg_n + \sum_{j=1}^{n-1} a_{n,j} \grb_{n,j} \in %\Psi_{\grd(\gra_n)}
\end{align*}
By Proposition~\ref{prop:attivazioni-equivalenti} it follows then
\begin{align*}
	\grb_1' := \gra_1 + \grg_1 \in \Psi_{\grd(\gra_1)}	\\
	\grb_2' := \gra_2 + \grg_2 + a_{1,1} \grb_1' \in \Psi_{\grd(\gra_2)} \\
	\grb_3' := \gra_3 + \grg_3 + a_{2,1} \grb_1' + a_{2,2} \grb_2' \in \Psi_{\grd(\gra_3)} \\
	\qquad \qquad \cdots \cdots \cdots \cdots \cdots \cdots \cdots \cdots \qquad \qquad \\
	\grb_n' := \gra_n + \grg_n + a_{n,1} \grb_1' + \ldots + a_{n,n-1} \grb_{n-1}' \in \Psi_{\grd(\gra_n)}
\end{align*}
In particular, for every $i = 1, \ldots, n$, it follows that $\grb'_i = \gra'_i + \grg'_i$ with $\grg'_i \in \mN \Psi_\m$ and $\gra'_i \in \mN \gra_1 + \ldots + \mN \gra_i$, and (\ref{eqn:formulone}) follows.

Let now $D \in \M \senza \m$, $\grb \in \Psi_D$ and $\grg \in \mN \Psi_\m$ as in (\ref{eqn:formulone}), and set $\gra = \grb - \grg$. Then by Lemma~\ref{lemma:differenza-attive} we have $\gra \in \Phi^+$, hence $\gra \in \Theta_{\m}$. On the other hand we have $w(\gra) \in \mN \Phi^+$, hence $\gra \in \Phi^+ \senza \Phi^+(w)$. Therefore $\gra \in \Theta_{\m} \senza \Phi^+(w)$, and by Theorem~\ref{teo:deadker} we get $D = \delta(\gra)$, which shows that $D \in \m(w)$.

ii) For the dimension formula, notice that the action of $G$ on $G/H$ induces an isomorphism of varieties
$$
	U \cap (wU^-w^{-1}) \times wB/H \lra BwB/H
$$
If $\scrO \subset BwB/H$ it follows that $\dim (\scrO) = l(w) + \dim (w^{-1}\scrO \cap B/H)$. On the other hand, since $B/H$ is a smooth affine toric variety, if $I \subset \scrD$ we have
$$
	\dim(\scrU_I) = \dim(B/H) - |\scrD \senza I| =  \dim(\scrU_\vuoto) + |I|,
$$
therefore the formula follows by the definition of $\M$.

For the rank formula, fix a point $x_\m \in \scrU_\m$. Notice that $x_\m$ is a standard base point in $(B \cap B^w)\scrU_\m$, as a homogeneous space under $B \cap B^w$. Indeed, if $x \in (B \cap B^w)\scrU_\m$ is a standard base point, then $Tx$ is a closed $T$-orbit by Lemma~\ref{lemma:B-rank}. On the other hand $\scrU_\m \subset (B \cap B^w)\scrU_\m$ is the unique closed $T$-orbit, therefore $Tx=Tx_\m$ and $x_m$ is a standard base point too.

It follows that $wx_\m \in \scrO$ is a standard base point as well. Indeed, since $x_\m \in B/H$ and $H\subset B$, we have $\Stab_G(x_\m) = bHb^{-1} = \Stab_B(x_\m)$, hence
$$
\Stab_B(wx_\m) = B \cap w \Stab_G(x_\m) w^{-1} = B \cap w \Stab_B(x_\m) w^{-1}.
$$
On the other hand the latter equals $w \Stab_{B \cap B^w} (x_\m)w^{-1}$, and since $\Stab_T(wx_\m) = w\Stab_T(x_\m) w^{-1}$ 
it follows that $wx_\m \in \scrO$ is a standard base point because $x_\m \in (B \cap B^w)\scrU_\m$ is so.

By Lemma~\ref{lemma:B-rank} we have then
$$
	\calX_B(\scrO) = \calX_T(w\scrU_\m) = w\calX_T(\scrU_\m),
$$
therefore $\rk(\scrO) = \dim(\scrU_\m) = \dim(\scrU_\vuoto) + |\m|$. On the other hand, by Lemma~\ref{lemma:B-rank} again, we have $\rk(B/H) =  \dim(\scrU_\vuoto)$, and the rank formula follows. Finally, the inclusion $w(\Phi_\m) \subset \calX_B(\scrO)$ follows by Proposition~\ref{prop:radici-reticoli}.
\end{proof}

Given $w \in W$ and $I\subset \scrD$ we will denote by $\m_{w,I}$ the minimal representative of $\scrO_{w,I}$ and by $\M_{w,I}$ the maximal representative of $\scrO_{w,I}$. 

\begin{definition}
Let $w \in W$ and $I\subset \scrD$. We say that $(w,I)$ is a \textit{reduced pair} if $I = \m_{w,I}$ (i.e., if $I \cap I(w) = \vuoto$), and we say that it is an \textit{extended pair} if $I = \M_{w,I}$ (i.e., if $I(w) \subset I$). Given $w \in W$ and $I \subset \scrD$, we will call $(w,\m_{w,I})$ and $(w,\M_{w,I})$ respectively the \textit{reduction} and the \textit{extension} of $(w,I)$.
\end{definition}

\begin{remark}\label{rem:pairs}
Given $I \subset \scrD$, we have by definition $\weak_I = \Theta_I \cap \grd^{-1}(I)$, whereas by Theorem~\ref{teo:parametrizzazione} we have the inclusion $\Theta_I \senza \weak_{\M_{w,I}} \subset \Phi^+(w)$. The following characterizations of reduced pairs and of extended pairs follow immediately:
\begin{itemize}
	\item[i)] The pair $(w,I)$ is reduced if and only if $\weak_I \subset \Phi^+(w)$, if and only if $\Phi^+_I \subset \Phi^+(w)$.
	\item[ii)] The pair $(w,I)$ is extended if and only if $\Theta_I \senza \weak_I \subset \Phi^+(w)$. 
\end{itemize}
\end{remark}

To summarize:

\begin{corollary} \label{cor:parametrizzazione}
The map $(w,I) \mapsto \scrO_{w,I}$ gives a parametrization by reduced (resp. extended) pairs
$$
	\{(w,I) \st \Phi^+_I \subset \Phi^+(w) \} \longleftrightarrow \scrB(G/H) \longleftrightarrow \{(w,I) \st \Theta_I \senza \weak_I \subset \Phi^+(w) \} .
$$
Moreover:
\begin{itemize}
	\item[i)] If $(w,I)$ is reduced, then $\rk (\scrO_{w,I}) = \rk(B/H) + |I|$.
	\item[ii)] If $(w,I)$ is extended, then $\dim (\scrO_{w,I}) = \dim(B/H) + l(w) - |\scrD \senza I|$.
\end{itemize}
\end{corollary}

\begin{remark}
The following properties easily follow from Remark~\ref{rem:pairs}:
\begin{itemize}
	\item[i)] All the pairs of the shape $(w,\vuoto)$ are reduced, and all the pairs of the shape $(w,\scrD)$ are extended. All the pairs of the shape $(w_0, I)$ are both reduced and extended.
	\item[ii)] If $(w,I)$ is reduced and if $J \subset I$, then $(w,J)$ is also reduced.
	\item[iii)] If $(w,I)$ is reduced (resp. extended) and if $w \preceq v$ (i.e. $w$ is a right subexpression of $v$), then $(v,I)$ is also reduced (resp. extended).
\end{itemize}
\end{remark}

\begin{corollary}	\label{cor: orbite chiuse}
Suppose that $(w,I)$ is reduced. Then $\scrO_{w,I}$ is closed in $G/H$ if and only if all of the following conditions hold:
\begin{itemize}
	\item[i)] $I = \vuoto$;
	\item[ii)] $\Phi^+(w) \subset \Psi$;
	\item[iii)] $\Psi_{\grd(\Phi^+(w))} = \Phi^+(w)$;
	\item[iv)] $\grd_{|\Phi^+(w)} : \Phi^+(w) \ra \scrD$ is injective.
% $\grd^{-1}(\grd(\Phi^+(w))) \cap \Psi = \Phi^+(w)$.
\end{itemize}
\end{corollary}

\begin{proof}
As a consequence of the main theorem in \cite{Re}, every closed $B$-orbit has minimal rank. Therefore every closed $B$-orbit in $G/H$ is of the shape $\scrO_{w,\vuoto}$ for some $w \in W$. On the other hand $B/H$ is a closed $B$-orbit, and by \cite[Proposition 2.2]{Br2} all closed $B$-orbits have the same dimension, therefore $\scrO_{w,\vuoto}$ is closed if and only if $\dim(\scrO_{w,\vuoto}) = \dim(B/H)$.

By Theorem~\ref{teo:parametrizzazione} the maximal representative of $\scrO_{w,\vuoto}$ is $\M = \grd(\Psi \senza \Phi^+(w))$, therefore $\dim(\scrO_{w,\vuoto}) = \dim(B/H)$ if and only if $|\scrD \senza \M| = l(w)$. On the other hand
$$
\scrD \senza \grd(\Psi \senza \Phi^+(w)) \subset \grd(\Psi \cap \Phi^+(w))
$$
and the latter has cardinality at most $l(w)$, therefore $|\scrD \senza \M| = l(w)$ if and only if $\Phi^+(w) \subset \Psi$ and $\scrD \senza \grd(\Psi \senza \Phi^+(w)) = \grd(\Phi^+(w))$ has cardinality $l(w)$. Assuming ii), the claim follows by noticing that the equality $\scrD \senza \grd(\Psi \senza \Phi^+(w)) = \grd(\Phi^+(w))$ is equivalent to iii), whereas the equality $|\grd(\Phi^+(w))| = l(w)$ is equivalent to iv).
\end{proof}

\begin{example}	\label{ex:maximal-rank}
In this example, we study the special case where $H$ contains a maximal torus of $G$, and we show how the theory developed in this section simplifies under this assumption. Assume that $T \subset H$.

\begin{itemize}
	\item[i)] Notice that $\dim(B/H) = |\Psi|$ and $\rk(B/H) = 0$. The first formula is clear, whereas the second one follows by Lemma \ref{lemma:B-rank} because $B/H$ contains a $T$-fixed point.
	\item[ii)] The restriction $\grd_{|\Psi} : \Psi \ra \scrD$ is bijective by Corollary~\ref{cor:delta-surjective}, therefore we may identify $\Psi$ and $\scrD$ (as we will in the following points), and regard $\grd_{|\weak}$ as a map $\weak \ra \Psi$ which extends the identity.
	\item[iii)] Let $(w,I)$ be a reduced pair, then by Corollary~\ref{cor: orbite chiuse} we have
	$$\scrO_{w,I} \text{ is closed} \iff I = \vuoto \text{ and } \Phi^+(w) \subset \Psi$$
	\item[iv)] Let $(w,I)$ be an extended pair and denote $\Phi^+(w,I) = \Phi^+(w) \senza (\Psi \senza I)$, then we have
	$$\dim \scrO_{w,I} = |\Psi|  + |\Phi^+(w,I)|.
	$$
	Indeed $\Psi \senza I = \Psi \cap (\Theta_I \senza \weak_I)$, therefore $\Psi \senza I \subset \Phi^+(w)$ by Remark \ref{rem:pairs} and the formula follows by by Corollary~\ref{cor:parametrizzazione}.
	\item[v)] Let $(w,I)$ be a reduced pair, then $\rk(\scrO_{w,I}) = |I|$ and
		$$\dim \scrO_{w,I} \geq |\Psi| + |\Phi^+_I|.$$
The rank formula follows by Theorem \ref{teo:parametrizzazione} thanks to i). The dimension formula follows by iii), thanks to the inclusion $\Phi^+_I \subset \Phi^+(w,\M)$ (where $(w,\M)$ denotes the extension of $(w,I)$): indeed $\Phi^+_I \subset \Phi^+(w)$ by Remark \ref{rem:pairs}, and if we regard $\grd_{|\weak}$ as a map $\weak \ra \Psi$ then by Proposition \ref{prop: I-rootsyst} we have $\Psi \cap \Phi^+_I = I \subset \M$.
\end{itemize}
\end{example}

\begin{example}[see \cite{Ti1}] \label{ex:timashev}
Let $U' = \prod_{\gra \in \Phi^+ \senza \grD} U_\gra$ be the derived subgroup of $U$, then $H = T U'$ is a spherical subgroup of $G$ which is contained in $B$. In particular $T \subset H$, therefore the discussion of Example \ref{ex:maximal-rank} applies. Notice that $\Psi = \grD$ is the set of the simple roots, therefore we can identify $\scrD$ with $\grD$ via the map $\grd$. Notice that $\weak = \Psi = \grD$: indeed by definition $\weak \subset \Phi^+$, and by Theorem \ref{teo:deadker} for all $\gra \in \weak$ there exists $\grb \in \Psi$ such that $\gra \leq \grb$.

Given $w \in W$ and $I \subset \grD$, we have by definition $I(w) = I \senza \Phi^+(w)$. Therefore by Remark \ref{rem:pairs} the pair $(w,I)$ is reduced if and only if $I \subset \Phi^+(w)$, whereas it is extended if and only if $\grD \senza \Phi^+(w) \subset I$. In particular if $(w,I)$ is reduced then we have $\rk (\scrO_{w,I}) = |I|$ and
$$
	\dim (\scrO_{w,I}) = l(w) + |I| + |\grD \senza \Phi^+(w)|.
$$	
\end{example}

%%%%%%%%%%%%%%%%%%%%%%%%%%%%%%%%%%%%
\subsection{Comparison with the wonderful case}\label{s:reduction}
%%%%%%%%%%%%%%%%%%%%%%%%%%%%%%%%%%%%

An important class of spherical subgroups of a reductive group $G$ is that of \textit{wonderful subgroups} (see e.g. \cite{BL}). This class plays a prominent role in the theory of spherical varieties, both in their classification and in the study of their geometry. Wonderful subgroups can be characterized in terms of their \textit{spherical roots}: by definition, the spherical roots of a spherical subgroup $H \subset G$ are a distinguished set of elements in the weight lattice $\calX(G/H)$ (see the references in Remark \ref{rem:spherical roots} for the definition), and $H$ is wonderful if and only if the corresponding spherical roots form a basis of $\calX(G/H)$.

To any spherical subgroup $H$ one may canonically associate a wonderful subgroup $\ol H$ containing $H$, called the \textit{spherical closure} of $H$. Given a spherical subgroup $H$ contained in $B$, the aim of this subsection is to compare the set of $B$-orbits in $G/H$ with that of $G/\ol H$, and to show that these two sets are canonically identified. We keep the notations introduced in Section~\ref{ssec:spherical}.

Restricting to the strongly solvable case, we say that a spherical subgroup $H \subset G$ contained in $B$ is wonderful if $w_0 \calX(G/H) = \mZ \grS$, where $\grS$ is the set introduced in Definition \ref{def:spherical-roots}. The fact that this definition agrees with the general one is a consequence of Proposition \ref{prop:B/H-toric} (see Remark \ref{rem:spherical roots}).

If $H$ is a spherical subgroup of $G$, then $N_G(H)$ acts by conjugation on $\calX(H)$. The spherical closure $\overline H$ of $H$ is by definition the kernel of this action, and we say that $H$ is \textit{spherically closed} if $H = \ol H$. Notice that we have inclusions $H \subset \ol H$ and $Z(G) \subset \overline H \subset N_G(H)$, where $Z(G)$ denotes the center of $G$. By \cite[Corollary 5.25]{Avd2}, it follows that if $H$ is contained in $B$, then $\ol H$ is also contained in $B$.

By a general theorem of Knop \cite[Theorem 7.5 and Corollary 7.6]{Kn3}, the spherical closure of a spherical subgroup of $G$ is wonderful. In the case of a strongly solvable spherical subgroup $H \subset G$, the converse also is true: $H$ is wonderful if and only if it is spherically closed (see \cite[Corollary 3.42]{Avd2}).

Let $H \subset G$ be a wonderful spherical subgroup contained in $B$, and denote $\grS_{G/H} = -w_0(\grS)$. Then by definition $\grS_{G/H}$ is a basis for $\calX(G/H)$, whose elements are simple roots of $G$.

\begin{lemma}\label{lemma:connected}
Suppose that $G$ is adjoint and let $H \subset G$ be a strongly solvable wonderful subgroup. Then $H$ is connected.
\end{lemma}

\begin{proof}
Let $H^\circ \subset H$ be the identity component. Since $H^\circ$ has finite index in $H$ and since $H$ has an open orbit on $G/B$, it follows that $H^\circ$ has an open orbit on $G/B$ as well, hence it is spherical.  The pull-back of rational functions along the projection $G/H^\circ \ra G/H$ identifies the weight lattice $\calX(G/H)$ with a sublattice of $\calX(G/H^\circ)$, which has finite index because $H/H^\circ \simeq \calX(G/H^\circ)/\calX(G/H)$ (see e.g. \cite[Lemma 2.4]{Ga}). 

Since $G$ is adjoint we have $\calX(G/H^\circ) \subset \calX(T)= \mZ\grD$. On the other hand $\calX(G/H) = \mZ \grS_{G/H}$ is generated by a subset of $\grD$. Since $\calX(G/H) \subset \calX(G/H^\circ) \subset \mZ \grD$ and since the first inclusion has finite index, it follows the equality $\calX(G/H) = \calX(G/H^\circ)$. The equality $H = H^\circ$ follows then by applying the isomorphism $H/H^\circ \simeq \calX(G/H^\circ)/\calX(G/H)$ once more.
\end{proof}

\begin{corollary}\label{cor:wonderful}
Suppose that $H \subset G$ is spherical and strongly solvable. Then the projection $G/H\to G/\ol H$ induces a bijection between $\scrB(G/H)$ and $\scrB(G/\ol H)$.
\end{corollary}

\begin{proof}
Denote $H' = H Z(G)$. Since, as we already recalled, $\ol H$ contains both $H$ and $Z(G)$, we have $H' \subset \ol H$. The projection $G/H\to G/H'$ induces a bijection between $\mathscr{B}(G/H)$ and $\mathscr{B}(G/H')$. Therefore we may replace $H$ with $H'$, and since $Z(G) \subset H'$ we may also assume that $G$ is adjoint. On the other hand by Lemma~\ref{lemma:connected} it follows that the quotient $\ol H/H$ is the image of a connected variety, hence it is connected as well and the claim follows by \cite[Lemma 3]{Br3}.
\end{proof}

%%%%%%%%%%%%%%%%%%%%%%%%%%%%%%%%%%%%%%%%%%
%%%%%%%%%%%%%%%%%%%%%%%%%%%%%%%%%%%%%%%%%%
\section{The Weyl group action on $\scrB(G/H)$}\label{s:W}
%%%%%%%%%%%%%%%%%%%%%%%%%%%%%%%%%%%%%%%%%%
%%%%%%%%%%%%%%%%%%%%%%%%%%%%%%%%%%%%%%%%%%

%%%%%%%%%%%%%%%%%%%%%%%%%%%%%%%%%%%%
\subsection{Preliminaries}	\label{ssec:prel-W-action}
%%%%%%%%%%%%%%%%%%%%%%%%%%%%%%%%%%%%

We denote by $M(W)$ the \textit{Richardson-Springer monoid}, namely the monoid generated by elements $m(s_\gra)$ ($\gra \in \grD$) with defining relations $m(s_\gra)^2 = m(s_\gra)$ for all $\gra \in \grD$ and the braid relations
$$m(s_\gra) m(s_\grb) m(s_\gra) \ldots = m(s_\grb) m(s_\gra) m(s_\grb) \ldots$$
for all $\gra, \grb \in \grD$, the number of factors on both sides being the order of $s_\gra s_\grb$ in $W$ (see \cite[3.10]{RS1}).

As a set, $M(W)$ is identified with $W$, and given $w \in W$ we will denote by $m(w) \in M(W)$ the corresponding element. Hence we may consider the Richardson-Springer monoid as the Weyl group with a different multiplication, defined by the following rule: if $w \in W$ and $\gra \in \grD$, then
$$
	m(s_\gra) m(w) = \left\{ \begin{array}{ll}
			m(s_\gra w) & \text{ if } l(s_\gra w) > l(w)\\
			m(w) & \text{ if } l(s_\gra w) < l(w)
	\end{array}\right.
$$
From a geometrical point of view, the multiplication on $M(W)$ coincides with the one defined on the Weyl group by the multiplication of Bruhat cells, namely we have $\ol{B wB w' B} = \ol{Bw''B}$, where $w'' \in W$ is the element defined by the equality $m(w'') = m(w) m(w')$. Sometime we will identify $M(W)$ and $W$ as sets, in that case we will denote the product in $M(W)$ of two elements $w, w' \in W$ by $w * w'$.

Given a spherical homogeneous variety $G/H$, both $W$ and $M(W)$ act on the set of $B$-orbits $\scrB(G/H)$. The action of $M(W)$ was defined by Richardson and Springer in the case of a symmetric homogeneous variety (see \cite[4.7]{RS1}), and the definition carries over without modifications to the spherical case. To define the action of $M(W)$ in the case  of a simple reflection, given $\gra \in \grD$ and $\scrO \in \scrB(G/H)$, consider $P_\gra \scrO$. This is an irreducible $B$-variety, and since $G/H$ is spherical it decomposes into finitely many $B$-orbits. The element $m(s_\gra) \cdot \scrO$ is defined then as the open $B$-orbit in $P_\gra  \scrO$. This definition extends to an action of the monoid $M(W)$ on $\scrB(G/H)$, and it allows to define a partial order  $\preceq$ on $\scrB(G/H)$ (called the \textit{weak order}) as follows:
$$
	\scrO \preceq \scrO' \text{ if and only if } \quad \exists w \in W \; : \; \scrO' = m(w) \cdot \scrO 
$$

The action of $W$ on $\scrB(G/H)$ is much more subtle than that of $M(W)$ and was defined by Knop \cite{Kn}. We recall the definition of this action in the case of a spherical subgroup $H \subset G$ contained in $B$, where the involved considerations turn out to be easier. By a case-by-case consideration (see \cite[Lemma 3.2]{Kn} and \cite[Lemma 5 iv)]{Br3}), the $B$-stable variety $P_\gra \scrO$ decomposes in the union of two $B$-orbits or in the union of three $B$-orbits. More precisely we have the following possibilities:
\begin{itemize}
	\item[U)] Suppose that $P_\gra \scrO = \scrO \cup \scrO'$ decomposes in the union of two orbits, and assume for simplicity that $\scrO$ is the open one. Then $\dim \scrO' = \dim \scrO -1$ and $\calX(\scrO') = s_\gra \calX(\scrO)$, and we define $s_\gra \cdot \scrO = \scrO'$.
	\item[T)] Suppose that $P_\gra \scrO = \scrO \cup \scrO' \cup \scrO''$ decomposes in the union of three orbits, and assume for simplicity that $\scrO$ is the open one. Then $\dim \scrO' = \dim \scrO'' = \dim \scrO-1$ and $s_\gra \calX(\scrO') = \calX(\scrO'') \subset \calX(\scrO) = s_\gra \calX(\scrO)$, where $\calX(\scrO)/\calX(\scrO') \simeq \mathbb Z$, and we define $s_\gra \cdot \scrO = \scrO$ and $s_\gra \cdot \scrO' = \scrO''$.
\end{itemize}

By \cite[Theorem 5.9]{Kn}, the $s_\gra$-actions defined above glue together into an action of $W$ on $\scrB(G/H)$. As a consequence of the previous analysis, notice that the rank of a $B$-orbit is invariant for this action, which agrees with the action of $W$ on the weight lattices $\calX(\scrO)$, where $\scrO \in \scrB(G/H)$.

\begin{example}
Consider the case $H = B$. Then $\scrB(G/B) = \{BwB/B \st w \in W \}$ is the set of the Schubert cells. Notice that all $B$-orbits of $G/B$ have rank 0: indeed any such orbit is homogeneous under the action of a suitable unipotent subgroup of $B$, hence every $B$-semiinvariant function on such orbit is constant. If $\gra \in \grD$ and $w \in W$, it follows that $P_\gra w B/B = BwB/B \cup Bs_\gra w B/B$ is of type (U), hence $s_\gra \cdot BwB/B = Bs_\gra w B/B$. Therefore the $W$ action on $\scrB(G/B)$ is induced by the action of $W$ on itself by right multiplication.
\end{example}

In some special cases the rank uniquely identifies the $W$-orbit, in particular this happens when the rank of a $B$-orbit is maximal and minimal. This is summarized in the following theorem, which holds for any spherical subgroup $H$: the case of maximal rank being due to Knop (see \cite[Theorem 6.2]{Kn}) and that of minimal rank to Ressayre (see \cite[Corollary 3.1 and Theorem 4.2]{Re}).

\begin{theorem}	\label{teo: ranghi speciali}
Let $\scrO, \scrO' \in \mathscr{B}(G/H)$ and suppose that $\rk \scrO = \rk \scrO' = \rk G/H$ (resp. $\rk \scrO = \rk \scrO' = \rk G - \rk H$). Then there exists $w \in W$ such that $\scrO' = w \cdot \scrO$.
\end{theorem}

%%%%%%%%%%%%%%%%%%%%%%%%%%%%%%%%%%%%%%%%%%%%%%%%
\subsection{Stabilizers for the $W$-action on $\scrB(G/H)$}
%%%%%%%%%%%%%%%%%%%%%%%%%%%%%%%%%%%%%%%%%%%%%%%%

We now describe the actions of $W$ and of $M(W)$ on $\scrB(G/H)$ in the case of a strongly solvable spherical subgroup $H \subset G$ contained in $B$, in terms of the combinatorial parametrization of Corollary~\ref{cor:parametrizzazione}.

In order to study the actions of $W$ and of $M(W)$ in terms of reduced and extended pairs, we take a closer look at the possible cases arising in the decomposition of the $B$-stable subsets $P_\gra \scrO_{w,I}$, where $\scrO_{w,I} \in \scrB(G/H)$ and $\gra \in \grD$.

\begin{lemma} \label{lemma:triple-T}
Let $w \in W$ and $\gra \in \grD$. Let $\grb$ be the unique positive root in the set $\{w^{-1}(\gra), -w^{-1}(\gra)\}$, then the following hold.
\begin{itemize}
	\item[i)] Suppose that $(w,I)$ is a reduced pair, then
$$
	s_\gra \cdot \scrO_{w,I} = \left\{ \begin{array}{ll}
			\scrO_{w,I} & \text{ if } \grb \in \weak_I \\
			\scrO_{s_\gra w,I} & \text{ otherwise }
	\end{array}\right.
$$
(where all the orbits above are expressed in terms of reduced pairs).
	\item[ii)] Suppose that $(w,I)$ is an extended pair, then
$$
	m(s_\gra) \cdot \scrO_{w,I} = \left\{ \begin{array}{ll}
			\scrO_{s_\gra \!* w, I\cup \{\grd(\grb)\}} & \text{ if } \grb \in \Theta_I \\
			\scrO_{s_\gra \!* w,I} & \text{ if } \grb \not \in \Theta_I
	\end{array}\right.
$$
(where all the orbits above are expressed in terms of extended pairs). In particular, $m(s_\gra) \cdot \scrO_{w,I} = \scrO_{w,I}$ if and only if $\grb \in \Phi^+(w) \senza (\Theta_I \senza \weak_I)$.
\end{itemize}
\end{lemma}

\begin{proof}
Let $w \in W$ and $I \subset \scrD$. By the results recalled in Section~\ref{ssec:prel-W-action} the $B$-stable subset $P_\gra \scrO_{w,I} = P_\gra w\scrU_I$ decomposes either in the union of two $B$-orbits which are permuted by the action of $s_\gra$, or in the union of three $B$-orbits, an open one fixed by $s_\gra$ and two of codimension one which are permuted by $s_\gra$.

Notice that if $\scrO_{v,J} \subset P_\gra \scrO_{w,I}$ is the open orbit, then it must be $v = s_\gra*w$, namely $v$ is the longest element between $w$ and $s_\gra w$. Indeed $P_\gra \scrO_{w,I} \cap B(s_\gra*w)B/H$ is a $B$-stable dense open subset, hence it must be $\scrO_{v,J} \subset  B(s_\gra*w)B/H$ because $P_\gra \scrO_{w,I}$ is irreducible and and the intersection of two open non-empty subsets therein is always non-empty. It follows that $\scrO_{v,J} \subset P_\gra \scrO_{w,I}$ is the open orbit if and only if $v$ has maximal length and $J$ has maximal cardinality among all the reduced pairs $(v',J')$ with $\scrO_{v',J'} \subset P_\gra \scrO_{w,I}$.

Recall that $w < s_\gra w$ in $W$ if and only if $w^{-1}(\gra) \in \Phi^+$, in which case $\Phi^+(s_\gra w) = \Phi^+(w) \cup \{w^{-1}(\gra)\}$. We will distinguish two different cases, depending on $s_\gra w < w$ or $s_\gra w > w$.

\textit{Case 1.} Suppose that $\grb = -w^{-1}(\gra)$  and denote $v = s_\gra w$,  so that $l(v) = l(w)-1$. The decomposition
$$
	P_\gra  = Bs_\gra \cup Bs_\gra B s_\gra = B s_\gra \cup B U_{-\gra}
$$
implies that $P_\gra \scrO_{w,I} = Bv \scrU_I \cup Bw U_{\grb} \scrU_I$. By Proposition~\ref{prop:Ualphaactionfinal} we get
$$
	U_\grb \scrU_I  = \left\{ \begin{array}{ll}
	\scrU_I & \text{if $\grb \in \Phi^+ \senza \Theta_I$}\\
	\scrU_I \cup \scrU_{I \cup \{\delta(\grb)\}} & \text{if $\grb \in \Theta_I \senza \weak_I$}\\
	\scrU_I \cup \scrU_{I \senza \{\delta(\grb)\}} & \text{if $\grb \in \weak_I$}
	\end{array} \right.
$$
By the discussion in Subsection \ref{ssec:prel-W-action} we have the following three possibilities, that we denote by U), T1) and T2) and that we describe in detail here below. Denote $\m = \m_{w,I}$ and $\M = \M_{w,I}$:
\begin{itemize}
	\item[U)] Suppose that $\grb \in \Phi^+ \senza \Theta_{\m}$ or that $\grb \in \Theta_{\m} \cap \weak_{\M} \senza \weak_\m$.\\
By Theorem \ref{teo:parametrizzazione}, we have in both these cases that $U_\grb \scrU_I \subset (B \cap B^w)\scrU_I$: indeed in the first case $U_\grb \scrU_\m = \scrU_\m$, whereas in the second case $U_\grb \scrU_\m = \scrU_\m\cup \scrU_{\m \cup \{\grd(\grb)\}}$, so that all subsets of $\scrD$ that arise by applying $U_\grb$ to $\scrU_\m$ are still representatives for $\scrO_{w,I}$. Therefore $BwU_\grb \scrU_I = \scrO_{w,I}$, and $P_\gra \scrO_{w,I} = \scrO_{v,I} \cup \scrO_{w,I}$ decomposes in the union of two orbits, and since $v < w$ the open orbit is $\scrO_{w,I}$. Therefore we have:
\begin{itemize}
	\item $\dim \scrO_{v,I} = \dim \scrO_{w,I} -1$;
	\item $\rk \scrO_{v,I} = \rk \scrO_{w,I}$;
	\item $\M_{v,I} = \M$;
	\item $\m_{v,I} =\m$.
\end{itemize}
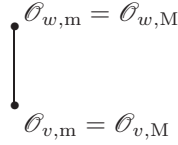
\[	\xy
0;/r5pc/:*\dir{*};
p+(0,0.2)="a",*+!UL{\scrO_{w,\m} = \scrO_{w,\M}},"a",
p+(0,-.5)*\dir{*}="c",
**\dir{-},*+!UL{\scrO_{v,\m}=\scrO_{v,\M}},"c",
\endxy
\]	
	\item[T1)] Suppose that $\grb \in \weak_\m$.\\
We have in this case $U_\grb \scrU_I = \scrU_I \cup \scrU_{I'}$, where $I' = I \senza \{\grd(\grb)\}$, and since $\m \not \subset I'$ Theorem~\ref{teo:parametrizzazione} shows that $\scrO_{w,I'} \neq \scrO_{w,I}$. Therefore $P_\gra \scrO_{w,I} = \scrO_{v,I} \cup \scrO_{w,I'} \cup \scrO_{w,I}$ decomposes in the union of three orbits. Since $v < w$ and $\m_{w,I'} = \m \senza \{\grd(\grb)\}$, it follows that $\scrO_{w,I}$ is the open orbit. On the other hand by Proposition \ref{prop:Ualphaactionfinal} we have $\grb \in \Theta_{I'}$, and since $\grb \not \in \Phi^+(v)$ by assumption, by Theorem \ref{teo:parametrizzazione} we get $\grd(\grb) \in I(v)$, namely $\scrO_{v,I} = \scrO_{v,I'}$. Therefore the following hold:
\begin{itemize}
	\item $\dim \scrO_{v,I} = \dim \scrO_{w,I'}  = \dim \scrO_{w,I} - 1$;
	\item $\rk \scrO_{v,I} = \rk \scrO_{w,I'} = \rk \scrO_{w,I} - 1$;
	\item $\M_{v,I} = \M$, $\M_{w,I'} = \M \senza \{\grd(\grb)\}$;
	\item $\m_{v,I} = \m_{w,I'} = \m \senza \{\grd(\grb)\}$.
\end{itemize}
\[	\xy
0;/r5pc/:*\dir{*};
p+(0,0.2)="a",*+!UL{\scrO_{w,\m} = \scrO_{w,\M}},"a",
p+(-.5,-.5)*\dir{*},**\dir{-},
p+(-1.3,-.5)="b",*+!UL{\scrO_{v,\m \senza \{\grd(\grb)\}}=\scrO_{v,\M}},"b",
p+(.5,-.5)*\dir{*}="c",
**\dir{-},*+!UL{\scrO_{w,\m \senza \{\grd(\grb)\}}=\scrO_{w,\M \senza \{\grd(\grb)\}}},"c",
\endxy
\]	

	\item[T2)] Suppose that $\grb \in \Theta_{\m} \senza \weak_\M$.\\
Then $U_\grb \scrU_I = \scrU_I \cup \scrU_{I'}$, where we set $I' = I \cup \{\grd(\grb)\}$, and by Theorem~\ref{teo:parametrizzazione} we have $\scrO_{w,I'} \neq \scrO_{w,I}$. Therefore $P_\gra \scrO_{w,I} = \scrO_{v,I} \cup \scrO_{w,I} \cup \scrO_{w,I'}$ decomposes in the union of three orbits as represented by the following diagram, and we have:
\begin{itemize}
	\item $\dim \scrO_{v,I} = \dim \scrO_{w,I}  = \dim \scrO_{w,I'} - 1$;
	\item $\rk \scrO_{v,I} = \rk \scrO_{w,I} = \rk \scrO_{w,I'} - 1$;
	\item $\M_{v,I} = \M_{w,I'} = \M \cup \{\grd(\grb)\}$;
	\item $\m_{v,I} = \m$, $\m_{w,I'} = \m \cup \{\grd(\grb)\}$.
\end{itemize}
\[	\xy
0;/r5pc/:*\dir{*};
p+(0,0.2)="a",*+!UL{\scrO_{w,\m \cup \{\grd(\grb)\}} = \scrO_{w,\M \cup \{\grd(\grb)\}}},"a",
p+(-.5,-.5)*\dir{*},**\dir{-},
p+(-1.3,-.5)="b",*+!UL{\scrO_{v,\m}=\scrO_{v,\M \cup \{\grd(\grb)\}}},"b",
p+(.5,-.5)*\dir{*}="c",
**\dir{-},*+!UL{\scrO_{w,\m}=\scrO_{w,M}},"c",
\endxy
\]	
\end{itemize}

\textit{Case 2.} Suppose now that $\grb = w^{-1}(\gra)$ and denote $v=s_\gra w$, so that $l(v) = l(w)+1$. The decomposition
$$
	P_\gra  = B \cup B s_\gra B = B \cup B s_\gra U_\gra.
$$
implies that $P_\gra \scrO_{w,I} = Bw\scrU_I \cup B v U_\grb \scrU_I$. Since $(w,I)$ is reduced we have $\grb \not \in \weak_I$ and $\grb \not \in \Theta_I \senza \weak_M$. By the discussion in Subsection \ref{ssec:prel-W-action} we have the following two possibilities, that we denote by U) and T) and that we describe in detail here below. Denote $\m = \m_{w,I}$, $\M = \M_{w,I}$, $\M' = \M_{v,I}$:
\begin{itemize}
	\item[U)] Suppose that $\grb \in \Phi^+ \senza \Theta_{\m}$ or $\grb \in \Theta_{\m} \cap \weak_{\M'} \senza \weak_\m$.\\
Then $U_\grb \scrU_I \subset (B \cap B^v)\scrU_I$, hence $BvU_\grb \scrU_I = \scrO_{v,I}$. Therefore $P_\gra \scrO_{w,I} = \scrO_{w,I} \cup \scrO_{v,I}$ decomposes in the union of two orbits as represented in the following diagram, where we have:
\begin{itemize}
	\item $\dim \scrO_{w,I} = \dim \scrO_{v,I} -1$;
	\item $\rk \scrO_{w,I} = \rk \scrO_{v,I}$;
	\item $\M_{v,I} = \M$;
	\item $\m_{v,I} = \m$.
\end{itemize}
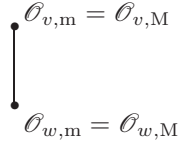
\[	\xy
0;/r5pc/:*\dir{*};
p+(0,0.2)="a",*+!UL{\scrO_{v,\m} = \scrO_{v,\M}},"a",
p+(0,-.5)*\dir{*}="c",
**\dir{-},*+!UL{\scrO_{w,\m}=\scrO_{w,\M}},"c",
\endxy
\]	
	\item[T)] Suppose that $\grb \in \Theta_{\m} \senza \weak_{\M'}$.\\
Then $U_\grb \scrU_I = \scrU_I \cup \scrU_{I'}$, where we set $I' = I \cup \{\grd(\grb)\}$, and by Theorem~\ref{teo:parametrizzazione} it follows $\scrO_{v,I'} \neq \scrO_{v,I}$. Therefore $P_\gra \scrO_{w,I} = \scrO_{w,I} \cup \scrO_{v,I} \cup \scrO_{v,I'}$ decomposes in the union of three orbits as represented by the following diagram, where we have:
\begin{itemize}
	\item $\dim \scrO_{v,I} = \dim \scrO_{w,I}  = \dim \scrO_{v,I'} - 1$;
	\item $\rk \scrO_{v,I} = \rk \scrO_{w,I} = \rk \scrO_{v,I'} - 1$;
	\item $\M_{v,I} = \M \senza \{\grd(\grb)\}$, $\M_{v,I'} = \M$;
	\item $\m_{v,I} = \m$, $\m_{v,I'} = \m \cup \{\grd(\grb)\}$.
\end{itemize}
\[	\xy
0;/r5pc/:*\dir{*};
p+(0,0.2)="a",*+!UL{\scrO_{v,\m \cup \{\grd(\grb)\}} = \scrO_{v,\M}},"a",
p+(-.5,-.5)*\dir{*},**\dir{-},
p+(-1.3,-.5)="b",*+!UL{\scrO_{w,\m}=\scrO_{w,\M}},"b",
p+(.5,-.5)*\dir{*}="c",
**\dir{-},*+!UL{\scrO_{v,\m}=\scrO_{v,\M \senza \{\grd(\grb)\}}},"c",
\endxy
\]	
\end{itemize}

The claims follow now by applying the definitions of the actions of $s_\gra$ and of $m(s_\gra)$ in all the possibilities presented above.
\end{proof}

\begin{remark}
In what follows, reduced pairs will play a main role in the understanding of the Weyl group action on $\scrB(G/H)$. In particular, we will be interested in the stabilizer of $\scrO_{w,I}$ under the action of $W$. If we restrict to the simple reflections which stabilize $\scrO_{w,I}$, we see by Lemma \ref{lemma:triple-T} that  $s_\gra \cdot \scrO_{w,I} =\scrO_{w,I}$ if and only if $-w^{-1}(\gra)$ is a weakly active root which stabilizes $I$ in the sense of Definition \ref{def:activated-roots}. Such simple reflections are parametrized by the set of weakly active roots $-w^{-1}(\gra) \in w^{-1}(\grD^-) \cap \weak_I$. Notice that
$$
w^{-1}(\grD^-) \cap \Phi^+_I = w^{-1}(\grD^-) \cap \weak_I = w^{-1}(\grD^-) \cap \grD_I.
$$
Indeed, we have the inclusions $\grD_I \subset \weak_I \subset \Phi^+_I \subset \Phi^+(w)$, and by definition $\grD_I$ is a basis for $\Phi_I$. If $\gra \in \grD$ and $-w^{-1}(\gra) \in \Phi^+_I$, It follows that there exist $\grb_1,\ldots, \grb_n \in \grD_I$ such that $-w^{-1}(\gra) = \grb_1+ \ldots + \grb_n$.  Therefore $-\gra = w(\grb_1) + \ldots + w(\grb_n)$, and since $\gra \in \grD$ we get $n=1$.
\end{remark}

\begin{corollary}
Let $(w,I)$ be an extended pair. The orbit $\scrO_{w,I}$ is minimal w.r.t. the weak order on $\scrB(G/H)$ if and only if
$$
	w^{-1}(\grD^-) \cap \Phi^+ \subset \Theta_I \senza \weak_I.
$$
\end{corollary}

\begin{proof}
Let $\gra \in \grD$. Then $P_\gra \scrO_{w,I} \cap B s_\gra wH/H \neq \vuoto$, therefore $P_\gra \scrO_{w,I}$ contains at least two $B$-orbits. In particular, it follows that $\scrO_{w,I}$ is minimal w.r.t. the weak order if and only if $m(s_\gra) \cdot \scrO_{w,I} \neq \scrO_{w,I}$ for all $\gra \in \grD$. On the other hand, by the analysis in the proof of Proposition \ref{lemma:triple-T} we have that $m(s_\gra) \cdot \scrO_{w,I} \neq \scrO_{w,I}$ if and only if either $w^{-1}(\gra) \in \Phi^+$ or $-w^{-1}(\gra) \in \Theta_I \senza \weak_I$. Therefore $\scrO_{w,I}$ is minimal if and only if $w^{-1}(\grD^-) \cap \Phi^+ \subset \Theta_I \senza \weak_I$.
\end{proof}

We now focus on the action of $W$ on $\scrB(G/H)$. First we show that the minimal representative of an orbit is a complete invariant for the action of $W$, namely it is invariant and it distinguishes the $W$-orbits. Then we will describe the stabilizers of the action in terms of reduced pairs.

\begin{theorem}	\label{teo:stabilizzatori1}
Let $(w,I)$, $(v,J)$ be reduced pairs. Then $\scrO_{w,I}$ and $\scrO_{v,J}$ are in the same $W$-orbit if and only if $I = J$, in which case we have $\scrO_{v,I} = vw^{-1} \cdot \scrO_{w,I}$.
\end{theorem}

\begin{proof}
By Lemma~\ref{lemma:triple-T} i) it follows that the minimal representative $I$ is an invariant for the action of $W$. To show that it uniquely determines the $W$-orbit, we show that if $(w,I)$ is reduced then $\scrO_{w_0,I} = w_0 w^{-1} \cdot \scrO_{w,I}$. In particular, this will imply that every $W$-orbit in $\scrB(G/H)$ contains a unique element which projects dominantly on $G/B$.

Suppose that $(w,I)$ is reduced and let $w_0 w^{-1} = s_{\gra_1} \cdots s_{\gra_n}$ be a reduced expression. Then $w_{i+1}^{-1}(\gra_i) \in \Phi^+$ for every $i=1, \ldots n$, where we denote
$$
	w_j = \left\{ \begin{array}{ll}
			w & \text{ if } j = n+1 \\
			s_{\gra_j} \cdots s_{\gra_n} w & \text{ if } 0 < j \leq n
	\end{array}\right.
$$
Therefore, for every $i \leq n$, we have $\Phi^+(w_i) = \Phi^+(w_{i+1}) \cup \{w_{i+1}^{-1}(\gra_i)\}$. Since $(w,I)$ is reduced, it follows that $\Phi^+_I \subset \Phi^+(w)$, thus $\Phi^+_I \subset \Phi^+(w_i)$ for all $i \leq n$. It follows that $(w_i,I)$ is reduced as well for all $i \leq n$, and by Lemma~\ref{lemma:triple-T} we get $\scrO_{w_{i-1},I} = s_{i-1} \cdot \scrO_{w_i,I}$. Combining all the steps we get $\scrO_{w_0,I} = w_0 w^{-1} \cdot \scrO_{w,I}$, and the last claim also follows.
\end{proof}

\begin{corollary}	\label{cor: Worbits}
The Weyl group $W$ has $2^{|\scrD|}$ orbits in $\scrB(G/H)$, and a complete set of representatives is given by the orbits which project dominantly on $G/B$, namely by the subsets of $\scrD$.
\end{corollary}

It follows by Theorem~\ref{teo:stabilizzatori1} that every $W$-orbit contains a distinguished element.

\begin{corollary}
Let $(w,I)$ be a reduced pair, then the element $w^{-1} \cdot\scrO_{w,I}$ depends only on $I$ and not on $w$.
\end{corollary}

Given a reduced pair $(w,I)$, we set $\scrO_I^\sharp = w^{-1} \cdot\scrO_{w,I}$. We now turn to the description of the stabilizers for the action of $W$ on $\scrB(G/H)$, and we prove the following theorem.

\begin{theorem}	\label{teo:stabilizzatori2}
Let $(w,I)$ be a reduced pair. Then $\Stab_W(\scrO_{w,I}) = wW_Iw^{-1}$.
\end{theorem}

Together with Theorem \ref{teo:stabilizzatori1}, the previous theorem gives a formula to compute the number of $B$-orbits in $G/H$ in terms of the root systems $\Phi_I$.

\begin{corollary}	\label{cor:numero-Borbite}
The number of $B$-orbits in $G/H$ is given by the formula
\[
	\sum_{I \subset \scrD} |W/W_I|.
\]
\end{corollary}

To prove Theorem \ref{teo:stabilizzatori2} we proceed by steps. First we prove one of the two inclusions.

\begin{lemma}	\label{lemma:stabilizzatori2}
Let $(w,I)$ be a reduced pair. Then $\Stab_W(\scrO_{w,I}) \subset wW_Iw^{-1}$.
\end{lemma}

\begin{proof}
Let $v \in \Stab_W(\scrO_{w,I})$, we proceed by induction on the length of $v$. If $l(v) = 1$, then by Lemma~\ref{lemma:triple-T} we have $v = s_\gra$ for some $\gra \in w(\weak_I)$, hence we assume $l(v) > 1$. Let $v = s_{\gra_m} \cdots s_{\gra_1}$ be a reduced expression, for $i \leq m$ we set $v_i = s_{\gra_i} \cdots s_{\gra_1}$ and $\scrO_i = v_i \cdot \scrO_{w,I}$. Suppose $\scrO_i \neq \scrO_{i-1}$ for all $i \leq m$, then Lemma~\ref{lemma:triple-T} implies $v \cdot \scrO_{w,I} = \scrO_{vw,I}$, hence $vw = w$ and $v=e$. Otherwise, let $n \leq m$ be such that $\scrO_n = \scrO_{n-1}$, and assume that $n$ is minimal with this property.

If $n = 1$ the claim follows by the inductive hypothesis, suppose $n > 1$. Denote $\gra = v_{n-1}^{-1} (\gra_n)$, then $v = v' s_\gra$, where $v' = s_{\gra_m} \cdots \hat{s}_{\gra_n} \cdots s_{\gra_1}$. By construction
$$
\scrO_{n-1} = v_{n-1} \cdot \scrO_{w,I} = \scrO_{v_{n-1} w,I}
$$
and $s_{\gra_n} \in \Stab_W(\scrO_{n-1})$, and it follows $s_\gra \in \Stab_W(\scrO_{w,I})$. On the other hand we have $w^{-1} v_{n-1}^{-1} (\gra_n) \in  \weak_I$ by Lemma~\ref{lemma:triple-T}, hence $\gra \in w(\weak_I)$ is of the desired shape and $s_\gra \in w W_I w^{-1}$. Finally $v' \in \Stab_W(\scrO_{w,I})$ and $l(v') < l(v)$, therefore we may apply the inductive hypothesis and it follows $v \in w W_I w^{-1}$.
\end{proof}

\begin{corollary}
Let $w_I \in W_I$ be the longest element, then we have $\scrO_I^\sharp = \scrO_{w_I,I}$.
\end{corollary}

\begin{proof}
Let $w \in W$ be such that $\scrO^\sharp_I = \scrO_{w,I}$. By the definition of $\scrO^\sharp_I$ we have the equality $\scrO_{w,I} = w^{-1} \cdot \scrO_{w,I}$, hence $w \in \Stab_W(\scrO_{w,I})$, and by Lemma~\ref{lemma:stabilizzatori2} we get $w \in W_I$. On the other hand $W_I$ contains a unique element such that $\Phi^+_I \subset \Phi^+(w)$, namely $w_I$.
\end{proof}

\begin{proof}[Proof of Theorem~\ref{teo:stabilizzatori2}]
Since $\scrO_{w,I} = w \cdot \scrO^\sharp_I$, by the previous corollary it is enough to show the equality $\Stab_W(\scrO^\sharp_I) = W_I$. The inclusion $\Stab_W(\scrO^\sharp_I) \subset W_I$ follows from Lemma~\ref{lemma:stabilizzatori2}, we show the other inclusion.

Let $v \in W_I$, we show $v \in \Stab_W(\scrO^\sharp_I)$ proceeding by induction on $l(v)$, where we regard $v$ as an element of $W$. If $l(v) = 1$, then we have $v = s_\gra$ for some $\gra \in \grD \cap \Phi^+_I \subset \grD_I$, hence $\gra = -w_I(\grb)$ for some other root $\grb \in \grD_I$. On the other hand $\grD_I \subset \weak_I$, and Lemma~\ref{lemma:triple-T} implies $s_\gra \in \Stab_W(\scrO_I^\sharp)$.

Suppose now that $l(v) > 1$, and let $v = s_{\gra_n} \cdots s_{\gra_1}$ be a reduced expression of $v$ as an element of $W$. If $i \leq n$, we denote $v_i = s_{\gra_i} \cdots s_{\gra_1}$ and $\scrO_i = v_i \cdot \scrO^\sharp_I$. Suppose $\scrO_i \neq \scrO_{i-1}$ for all $i \leq n$, then Lemma~\ref{lemma:triple-T} implies that $v \cdot \scrO_{w_I,I} = \scrO_{vw_I,I}$ and that $(vw_I,I)$ is a reduced pair. On the other hand $vw_I \in W_I$, and $w_I$ is the unique element in $W_I$ such that $\Phi^+_I \subset \Phi^+(w_I)$. Therefore $vw_I = w_I$, which is absurd since $v \neq e$. Therefore $s_{\gra_k} \in \Stab_W(\scrO_{k-1})$ for some $k \leq n$. Assume that $k$ is minimal with this property, then Lemma~\ref{lemma:triple-T} implies
$$
	\scrO_{k-1} = v_{k-1} \cdot\scrO_I^\sharp = \scrO_{v_{k-1} w_I,I}.
$$

By construction $s_{\gra_k} \in \Stab_W(\scrO_{k-1})$, hence $-w_I v^{-1}_{k-1} (\gra_k) \in \weak_I$  by Lemma~\ref{lemma:triple-T}. Denote $\gra = v^{-1}_{k-1} (\gra_k)$, then $s_\gra \in \Stab_W(\scrO_I^\sharp)$ and $\gra \in \Phi^+_I$, hence $s_\gra \in W_I$. On the other hand $v = s_{\gra_n} \cdots \hat{s}_{\gra_k} \cdots s_{\gra_1} s_\gra$, therefore $s_{\gra_n} \cdots \hat{s}_{\gra_k} \cdots s_{\gra_1} \in W_I$ and applying the inductive hypothesis it follows $v \in \Stab_W(\scrO_I^\sharp)$.
\end{proof}

%%%%%%%%%%%%%%%%%%%%%%%%%%%%%%%%%%%%%%%%%%%%%%%%
\subsection{Reduced pairs and weight polytopes}
%%%%%%%%%%%%%%%%%%%%%%%%%%%%%%%%%%%%%%%%%%%%%%%%

Building upon Theorem~\ref{teo:stabilizzatori2}, we now produce a combinatorial model for the action of $W$ on $\scrB(G/H)$ in terms of weight polytopes.

Consider the Weyl group $W_I$. The system of positive roots $\Phi^+_I$ induces a Bruhat order $\leq_I$ on $W_I$, which is compatible with the restriction of the Bruhat order $\leq$ on $W$ in the following sense.

\begin{lemma} \label{lem: I-bruhat}
Let $(w,I)$ be a reduced pair. Let $v_1, v_2 \in W_I$ be such that $v_1 \leq_I v_2$, then $wv_2 \leq wv_1$. In particular, the left coset $wW_I$ possesses a unique minimal element and a unique maximal element with respect to $\leq$, namely $ww_I$ and $w$.
\end{lemma}

\begin{proof}
Recall that $G_I \subset G$ is the reductive subgroup generated by $T$ together with the root spaces $U_\gra$ with $\gra \in \Phi_I$, and $B_I \subset G_I$ is the Borel subgroup $B \cap G_I$.
%Then we have an embedding of flag varieties $G_I/B_I \subset G/B$.
As $v_1 \leq_I v_2$, it follows $w_I v_2 \leq_I w_I v_1$, hence $w_I v_2 \in \ol{B_I w_I v_1 B_I}$.

Since $\Phi_I^+ \subset \Phi^+(w) \cap \Phi^+(w_I)$, it follows $\Phi^+(ww_I) \cap \Phi^+_I = \vuoto$, hence we get the equality $Bww_IT = Bww_I B_I$. Therefore
$$
w v_2 \in Bw w_I B_I w_I v_2 B_I \subset \ol{Bw w_I B_I w_I v_1 B_I} \subset \ol{Bwv_1 B_I} \subset \ol{Bwv_1 B} 
$$
and the claim follows.
\end{proof}

Denote by $\grL$ the weight lattice of $T$ and let $\grl \in \grL$ be a regular dominant weight. Denote $P = \conv(W\grl)$ the weight polytope of $\grl$ in $\grL_\mQ$, then the vertices of $P$ correspond bijectively to the elements of $W$. By a \textit{subpolytope} of $P$ we mean the convex hull of a subset of vertices of $P$. Denote $\scrS(P)$ the set of subpolytopes of $P$, then the Weyl group acts naturally on $\scrS(P)$. Given $I \subset \scrD$ we set $\scrS_I = \conv(W_I \grl)$, and for a reduced pair $(w,I)$ we set
$$
	\scrS_{w,I} = w \scrS_I = \conv (wW_I \grl) = \conv\big(vw (\grl) \st v \in \Stab_W(\scrO_{w,I})\big)
$$
(where the last equality follows from Theorem~\ref{teo:stabilizzatori2}). Denote moreover
$$
	\scrC_{w,I} = \cone(-w(\weak_I)) = \cone(-w(\Phi^+_I)) = -w\big(\mQ \Psi_I \cap \mQ_{\geq 0} \grD\big)
$$

\begin{theorem}	\label{teo:subpolytopes}
The map $\scrO_{w,I} \mapsto \scrS_{w,I}$ is a $W$-equivariant embedding of $\scrB(G/H)$ into $\scrS(P)$. Moreover we have the equality
$$\scrS_{w,I} = P \cap (\scrC_{w,I}+w\grl),$$
and in particular $\dim (\scrS_{w,I}) = \rk(\Phi_I)$.
\end{theorem}

\begin{proof}
Let $I \subset \scrD$ and denote by $\scrB_I(G/H)$ the corresponding $W$-orbit in $\scrB(G/H)$. We claim that the map $\scrO_{w,I} \mapsto \scrS_{w,I}$, regarded as a map $\scrB_I(G/H) \ra \scrS(P)$, is injective and $W$-equivariant. Indeed, since $(w,I)$ is reduced, by construction we have
$$
\Stab_W(\scrO_{w,I}) = \Stab_W(\scrS_{w,I}).
$$
Since $\scrB_I(G/H)$ is a single $W$-orbit, this shows that the map is injective, and since by Lemma~\ref{lemma:triple-T} the actions of the simple reflections on $\scrO_{w,I}$ and on $\scrS_{w,I}$ coincide, it shows that it is $W$-equivariant as well.

In particular the map $\scrB(G/H) \ra \scrS(P)$ is $W$-equivariant and we need only to show the injectivity, namely that $I$ is determined by $\scrS_{w,I}$. First of all, notice that $w$ is determined by $\scrS_{w,I}$. Indeed, by Lemma~\ref{lem: I-bruhat}, $w$ is maximal in $wW_I$ w.r.t. the Bruhat order, and is uniquely determined by this property. It follows that $w\grl$ is the unique minimal vertex of $\scrS_{w,I}$ w.r.t. the dominance order, hence $w$ is uniquely determined by $\scrS_{w,I}$. Then the equality $\scrS_{w,I} = \conv (wW_I \grl)$ implies that the set of differences $\{v\grl - w\grl \st v\grl \in \scrS_{w,I} \}$ generates the semigroup $w(\mN\Phi^+_I)$, hence we recover $I$ from $w$ and $\scrS_{w,I}$ thanks to Proposition~\ref{prop: I-rootsyst}. The last claim also follows.
\end{proof}

We conclude this subsection by showing that the parametrization of orbits via subpolytopes of $P$ of Theorem~\ref{teo:subpolytopes} is compatible with the Bruhat order on $\scrB(G/H)$ in the following sense.

\begin{proposition}	\label{prop: bruhat-compatibile}
Let $(w,I)$ and $(v,J)$ be reduced pairs and suppose that $\scrS_{v,J} \subset \scrS_{w,I}$. Then we have $\scrO_{v,J} \subset \ol{\scrO_{w,I}}$.
\end{proposition}

The proposition is an easy consequence of the following lemma.

\begin{lemma}\label{lemma:closure}
Let $(w,I)$ be a reduced pair and let $J \subset \M_{w,I}$. If $v \in wW_I$, then $\scrO_{v,J} \subset \ol{\scrO_{w,I}}$.
\end{lemma}

\begin{proof}
We will make use of the following fact, which follows by the description of the $T$-stable curves in $G/B$ (see \cite[Proposition 3.9]{Ca2}): if $x \in W$ and $\gra \in \Phi^+(x)$, then $x s_\gra \in \ol{B x U_\gra}$.

Let $\gra_1, \ldots, \gra_m \in \grD_I$ be such that $w^{-1} v = s_{\gra_1} \cdots s_{\gra_m}$ is a reduced expression and, for $i \leq m$, denote $w_i = w s_{\gra_1} \cdots s_{\gra_i}$. Then by Lemma~\ref{lem: I-bruhat} we have the following chain in the Bruhat order
$$
	v = w_n < \ldots < w_i < \ldots < w_1 < w.
$$
Denote $\M = \M_{w,I}$. As $J \subset \M$, it follows $\scrO_{v,J} \subset \ol{\scrO_{v,\M}}$, therefore to prove the theorem it is enough to consider the case $J = \M$. In particular, we have the inclusion $I \subset J$, hence $\grD_I \subset \weak_I \subset \weak_J$. It follows then by Proposition~\ref{prop:radici-reticoli} that $\ol{\scrU_J} = \ol{B_J \scrU_J}$ is $U_{\gra_i}$-stable for all $i$. In particular we have that $U_{\gra_1} U_{\gra_2} \cdots U_{\gra_n} \scrU_J \subset \ol{\scrU_J}$, and by the remark at the beginning of the proof we get the inclusions
\[
	\ol{\scrO_{w,I}} = \ol{Bw \scrU_J} = \ol{Bw U_{\gra_1} \cdots U_{\gra_n} \scrU_J} \supset \ol{Bw_1 U_{\gra_2} \cdots U_{\gra_n} \scrU_J} \supset \cdots \supset \ol{Bw_n \scrU_J} \supset Bv \scrU_J.
\qedhere
\]
\end{proof}

\begin{proof}[Proof of Proposition~\ref{prop: bruhat-compatibile}]
Notice that $vW_J \subset wW_I$, hence $w^{-1}v \in W_I$ implies $W_J \subset W_I$. In particular it follows $\Phi^+_J \subset \Phi^+_I$, therefore
$$
	J = \delta(\Phi^+_J \cap \weak) \subset \delta(\Phi^+_I \cap \weak) = I,
$$
and the claim follows by Lemma~\ref{lemma:closure}.
\end{proof}

%%%%%%%%%%%%%%%%%%%%%%%%%%%%%%%%%%%%%%%%%%
%%%%%%%%%%%%%%%%%%%%%%%%%%%%%%%%%%%%%%%%%%
\section{A bound for the number of $B$-orbits in $G/H$.}	\label{s:Knop-bound} 
%%%%%%%%%%%%%%%%%%%%%%%%%%%%%%%%%%%%%%%%%%
%%%%%%%%%%%%%%%%%%%%%%%%%%%%%%%%%%%%%%%%%%

A conjecture of Knop states that the homogeneous variety $G/TU'$ of Example~\ref{ex:timashev} has the largest number of $B$-orbits among all the homogeneous spherical $G$-varieties. In this section we prove such bound in the setting of solvable spherical subgroups. That is, we prove the following theorem.

\begin{theorem} \label{teo:knop-conj}
The spherical variety $G/TU'$ has the largest number of $B$-orbits among the homogeneous spherical varieties $G/H$ with $H$ a solvable subgroup of $G$. 
\end{theorem}

The set $\scrB(G/TU')$ is nicely described in terms of faces $\scrF(P)$ of the weight polytope $P$ of a regular dominant weight. Indeed, in this case the embedding of Theorem~\ref{teo:subpolytopes} induces a $W$-equivariant bijection between $\scrB(G/TU')$ and $\scrF(P)$ (see also \cite[Proposition 3.5]{Ti1} and \cite[Section 3]{Vi1} for a description of $\scrF(P)$ and of $\scrB(G/TU')$ in terms of the $W$-action).

We first prove Theorem~\ref{teo:knop-conj} in the basic case of a maximal rank strongly solvable spherical subgroup of $G$.

\begin{lemma} \label{lemma:Knop-maximal-rank}
The spherical variety $G/TU'$ has the largest number of $B$-orbits among the homogeneous spherical varieties $G/H$ with $H$ a strongly solvable subgroup of $G$ of maximal rank.
\end{lemma}

\begin{proof}
Suppose that $\pi(I) \subset \Psi$ is connected. By Proposition~\ref{prop:intervalli} $\Phi_I$ and $\Phi_I'$ are both irreducible of rank $|I|$. In particular, if $\Phi_I'$ is of type $\sfA$, then it follows the inequality $|W_I'| \leq |W_I|$. Otherwise, if $\Phi_I'$ is not of type $\sfA$, then Proposition~\ref{prop:intervalli} shows that $W'_I = W_I$.

Consider now an arbitrary subset $I \subset \Psi$ and let $I = I_1 \cup \ldots \cup I_m$ correspond to the decomposition of $\pi(I)$ into connected components. Then $\Phi'_I = \Phi'_{I_1} \times \ldots \times \Phi'_{I_m}$, and for all $j=1, \ldots, m$ it holds $|W'_{I_j}| \leq |W_{I_j}|$. Since $\Psi$ is linearly independent, $\Phi_{I_h} \cap \Phi_{I_k} = \vuoto$ for all $h \neq k$, hence $W_{I_h} \cap W_{I_k} = \{e\}$ for all $h \neq k$, and it follows that
$$
|W'_I| = |W'_{I_1} \times \ldots \times W'_{I_m}| \leq |W_{I_1} \times \ldots \times W_{I_m}| \leq |W_I|.$$
Therefore by Corollary \ref{cor:numero-Borbite} we get
\[
	|\scrB(G/H)| = \sum_{I \subset \Psi} |W/W_I| \leq \sum_{I \subset \Psi} |W/W'_I| \leq  \sum_{I' \subset \grD} |W/W_{I'}| = |\scrB(G/H')|.	\qedhere
\]

%Conversely, suppose that $\Phi_I^1 \subset \Phi_I$ is an irreducible component and let $\grb, \grb' \in I \cap \Phi_I^1$. Since $\weak_I$ contains a basis of $\Phi_I$, there are $\gra_1, \ldots, \gra_n \in \weak_I \cap \Phi_I^1$ such that $(\gra_i, \gra_{i+1}) \neq 0$ for all $i=0, \ldots, n$ (where we set $\gra_0 = \grb$ and $\gra_{n+1} = \grb'$).

%If indeed $\{\grb, \grb'\} \subset I$ is an edge of $\mathcal G(I)$, then either $\grb + \grb' \in \Phi_I$ or $\grb - \grb' \in \Phi_I$, therefore $\grb$ and $\grb'$ belong to the same irreducible component of $\Phi_I$. Conversely, suppose that $\Phi_I^1 \subset \Phi_I$ is an irreducible component and let $\grb, \grb' \in I \cap \Phi_I^1$. Since $\weak_I$ contains a basis of $\Phi_I$, there are $\gra_1, \ldots, \gra_n \in \weak_I \cap \Phi_I^1$ such that $(\gra_i, \gra_{i+1}) \neq 0$ for all $i=0, \ldots, n$ (where we set $\gra_0 = \grb$ and $\gra_{n+1} = \grb'$).
\end{proof}

\begin{proof}[Proof of Theorem \ref{teo:knop-conj}]
We reduce the proof of the theorem to the case of a strongly solvable spherical subgroup of $G$, which is treated in the previous lemma.

Suppose that $H$ is a solvable spherical subgroup and denote by $H^\circ \subset H$ the identity component. Clearly $G/H^\circ$ has a larger number of $B$-orbits than $G/H$, and since it is solvable and connected $H^\circ$ is contained inside a Borel subgroup of $G$. Therefore we may assume that $H$ is strongly solvable.

Suppose now that $H$ is a strongly solvable spherical subgroup, assume that $H \subset B$ and that $T \cap H$ contains a maximal torus $T_H \subset H$. Recall the restriction $\tau : \Psi \ra \calX(T_H)$, following the discussion at the end of Section \ref{s:toricvar} we identify the maps $\tau : \Psi \ra \ol\Psi$ and $\grd : \Psi \ra \scrD$. If $D \in \scrD$, denote $\Psi_D = \grd^{-1}(D)$. As in \cite[\S 2]{Avd1}, define a partial order on $\scrD$ as follows: $D \leq D'$ if there are $\grb \in \Psi_D$ and $\grb' \in \Psi_{D'}$ with $\grb \leq \grb'$. Equivalently, by \cite[Proposition 1]{Avd1}, we have $D \leq D'$ if and only if there exists $\gra \in \Phi^+$ such that $\Psi_D + \gra \subset \Psi_{D'}$.

Pick a representative $\grb_D \in \Psi_D$ for all $D \in \scrD$. Set $\scrD^1 = \scrD$ and let $\scrD_1$ be the set of the maximal elements in $\scrD$. We define inductively
$$
 	\scrD^{i+1} = \scrD^i \senza \bigcup_{D \in \scrD_i}\grd(F(\grb_D)),	
$$
and we define $\scrD_{i+1}$ as the set of the maximal elements in $\scrD^{i+1}$. Let $p$ be the maximum such that $\scrD_p$ is not empty and set $\scrD_* = \bigcup_{k \leq p} \scrD_k$. Finally, define
$$\Psi' = \bigcup_{D \in \scrD_*} F(\grb_D).$$

We devote the following paragraphs to show that $\grd_{|\Psi'} : \Psi' \ra \scrD$ is bijective. Indeed, it is surjective by construction. Suppose by contradiction that it is not injective and let $\gra, \gra' \in \Psi'$ be such that $\gra \neq \gra'$ and $\grd(\gra) = \grd(\gra')$. Let $D, D' \in \scrD_*$ and let $\grg, \grg' \in \Psi$ be maximal elements such that
$$
	\gra \leq \grb_D \leq \grg \qquad \text{ and } \qquad  	\gra' \leq \grb_{D'} \leq \grg'.
$$

Since the maps $\tau : \Psi \ra \ol\Psi$ and $\grd : \Psi \ra \scrD$ are canonically identified and since $\grd(\gra) = \grd(\gra')$, by \cite[Lemma~5]{Avd1} it follows that $\grb_D \neq \grb_{D'}$, hence $D \neq D'$. Similarly, by \cite[Lemma~8 and Lemma~11]{Avd1} it follows that $\grd(\grg) = \grd(\grg')$ and $\grg-\gra = \grg' - \gra'$.

By Proposition \ref{prop:attivazioni-equivalenti} we have $\grb_D - \gra  \in \weak$ and $\grb_{D'}-\gra' \in \weak$. Setting moreover $\tilde \grb_D = \gra' + \grb_D - \gra$ and $\tilde \grb_{D'} = \gra + \grb_{D'} - \gra'$, by the same proposition we see that $\tilde \grb_D$ and $\tilde \grb_{D'}$ are active roots, and they satisfy the equalities
$$
\grd(\tilde \grb_D) = \grd(\grb_D) = D, \qquad \qquad \grd(\tilde \grb_{D'}) = \grd(\grb_{D'}) = D'. 
$$
Since $\grg-\gra = \grg' - \gra'$, we have the inequalities
$$
	\gra \leq \tilde \grb_{D'} \leq \grg \qquad \text{ and } \qquad  	\gra' \leq \tilde \grb_D \leq \grg'.
$$

Notice that $\grb_D$ and $\tilde \grb_{D'}$ are comparable by Lemma \ref{lemma:radici-adiacenti}: indeed $\gra \in F(\grb_D) \cap F(\tilde \grb_{D'})$, thus $\supp(\grb_D) \cap \supp(\tilde \grb_{D'})$ is not empty, and $\supp(\grb_D) \cup \supp(\tilde \grb_{D'})$ is connected. Therefore we have either that $\tilde \grb_{D'} \in F(\grb_D)$ or that $\tilde \grb_D \in F(\grb_{D'})$. It follows that either $D' \in \grd(F(\grb_D))$ or $D \in \grd(F(\grb_{D'}))$, and by the definition of $\scrD_*$ we get $D = D'$, a contradiction.

Therefore we proved the bijectivity of the map $\grd_{|\Psi'} : \Psi' \ra \scrD$, and it follows that $\tau_{|\Psi'} : \Psi' \ra \ol \Psi$ is bijective as well. By making use of the properties of $\Psi'$, we now construct a strongly solvable spherical subgroup $H'$ containg $T$ such that $|\scrB(G/H)| \leq |\scrB(G/H')|$.

Notice that the elements of $\Psi'$ are linearly independent: indeed $\tau_{|\Psi'} : \Psi' \ra \ol \Psi$ is bijective, and $\ol\Psi$ is linearly independent in $\calX(T_H)$ by \cite[Theorem 1]{Avd1}. Let $\gra, \gra' \in \Phi^+$ and suppose that $\gra + \gra' \in \Psi'$. Since $\goh \subset \gog$ is a subalgebra and since by definition $\gou_{\gra+\gra'} \not \subset \goh$, it follows that either $\gou_\gra \not \subset \goh$ or $\gou_{\gra'} \not \subset \goh$. It follows that either $\gra \in \Psi$ or $\gra' \in \Psi$, hence either $\gra \in \Psi'$ or $\gra' \in \Psi'$. Therefore
$$H' = T \prod_{\gra \in \Phi^+ \senza \Psi'} U_\gra$$
is a strongly solvable subgroup of $G$, and by \cite[Theorem 1]{Avd1} it is spherical because $\Psi'$ is linearly independent. 

We claim that $|\scrB(G/H)| \leq |\scrB(G/H')|$. Denote indeed $\scrD' = \Div_T(B/H')$. Then by Corollary~\ref{cor:delta-surjective} we have a bijection $\Psi' \ra \scrD'$, and by the discussion above the restriction gives a bijection $\Psi' \ra \ol \Psi$. Thus we have a bijection $\scrD' \ra \scrD$, and by Corollary \ref{cor: Worbits} we get that $W$ acts on $\scrB(G/H)$ and on $\scrB(G/H')$ with the same number of orbits. On the other hand, if $I \subset \scrD$ and $\Phi'_I = \mZ \Psi'_I \cap \Phi$ is the root system associated to the corresponding $T$-orbit in $B/H'$, then by construction we have $\Psi'_I \subset \Psi_I$, hence $\Phi'_I \subset \Phi_I$. Denoting by $W'_I$ be the Weyl group of $\Phi'_I$, it follows that $W'_I \subset W_I$, thus by Corollary \ref{cor:numero-Borbite} we get
$$
	|\scrB(G/H)| = \sum_{I \subset \scrD} |W/W_I| \leq \sum_{I \subset \scrD} |W/W'_I| = |\scrB(G/H')|.
$$
On the other hand by construction $H'$ is a strongly solvable spherical subgroup of $G$ of maximal rank, therefore $|\scrB(G/H')| \leq |\scrB(G/TU')|$ by Lemma~\ref{lemma:Knop-maximal-rank}.
\end{proof}

\end{document}